\documentclass[oneside]{compositio}

\usepackage[theoremfont]{newtxtext}
\usepackage[varvw,smallerops]{newtxmath}

\usepackage{verbatim}

\usepackage{tikz-cd} 

\usepackage{array}
\newcolumntype{P}[1]{>{\raggedright\arraybackslash}p{#1}}
\usepackage{tabularx}
\usepackage{float}

\usepackage{hyperref}
\hypersetup{
    colorlinks=false, 
    linktoc=all,     
    linkcolor=black,  
}

\usepackage{placeins}

\newtheorem{thm}{Theorem}[section]
\newtheorem{cor}[thm]{Corollary}
\newtheorem{prop}[thm]{Proposition}
\newtheorem{lem}[thm]{Lemma}

\theoremstyle{definition}
\newtheorem{defn}[thm]{Definition}

\theoremstyle{remark}
\newtheorem{exmp}[thm]{Example}
\newtheorem{rem}[thm]{Remark}

\newcommand{\Q}{\mathbb{Q}} 
\newcommand{\C}{\mathbb{C}} 
\renewcommand{\P}{\mathbb{P}}
\newcommand{\Z}{\mathbb{Z}}
\newcommand{\R}{\mathcal{R}}

\newcommand{\E}{\mathcal{E}}
\newcommand{\Oo}{\mathcal{O}}
\renewcommand{\L}{\mathbb{L}}
\newcommand{\Ll}{\mathcal{L}}
\newcommand{\M}{\mathcal{M}}
\newcommand{\oM}{{\mkern5mu\overline{\mkern-5mu\M\mkern-1.5mu}\mkern2mu}}

\DeclareMathOperator{\Bl}{Bl}
\DeclareMathOperator{\Spec}{Spec}
\DeclareMathOperator{\Sm}{\mathbf{Sm}}
\DeclareMathOperator{\Rings}{\mathbf{Rings}}

\begin{document}

\title{Cobordism-valued intersection theory on $\oM_{0,n}$}

\author{Benjamin Ellis-Bloor}

\email{benjaminellis-bloor2026@u.northwestern.edu}

\address{Department of Mathematics, Northwestern University, Evanston,
  IL 60208}


\begin{abstract}
We calculate the genus zero cobordism-valued Gromov-Witten invariants of a point by refining the string equation on $\oM_{0,n}$ from the Chow ring to algebraic cobordism. This gives inductive formulas for cobordism-valued psi-class intersections on $\oM_{0,n}$, and in particular the cobordism classes $[\oM_{0,n}]$, and for their images in $K$-theory. Explicit formulas are given up to $n = 8$.
\end{abstract}

\maketitle

\section{Introduction}

The idea of refining Gromov-Witten invariants from rational cohomology (or the rational Chow groups) to instead take values in complex cobordism goes back to Kontsevich \cite[Section 1.5]{Kontsevich}. More recently, Abouzaid and Bai \cite{AbouzaidBai} have constructed complex cobordism valued Gromov-Witten invariants for general symplectic manifolds. Earlier work of Coates and Givental \cite{CoatesGivental} proposed a definition of Gromov-Witten invariants valued in rational complex cobordism. 

This paper focuses on calculations of these invariants (in the integral case) in the simplest
possible situation, where the target is a point and the genus is zero. That
is, our aim is to calculate the cobordism class
$[\oM_{0,n}]$ of the moduli space of stable genus zero curves with $n$ marked points (here $n \ge 3$), as an element of the cobordism ring. We will primarily work in the context of
algebraic cobordism $\Omega^{*}$ constructed by Levine-Morel \cite{AlgCob}
and Levine-Pandharipande \cite{LevinePandharipande}, which is the analogue of complex
cobordism $\text{MU}^{*}$ for smooth algebraic varieties.

Algebraic cobordism is defined over any field $k$ of characteristic
zero, which we assume in this paper to be algebraically closed. Both
$\text{MU}^{*}$ and $\Omega^{*}$ are governed by the universal formal group
law, and both $\text{MU}^{*}(\text{pt})$ and $\Omega^{*}(\Spec k)$ are
isomorphic to the Lazard ring $\L^{*} \cong \Z[u_{1},u_{2},u_{3},...]$
classifying one-dimensional, commutative formal group laws, where in
algebraic cobordism $u_{i}$ has degree $-i$, and in complex cobordism
$u_{i}$ has degree $-2i$. In particular, when working over $k = \C$,
there is a realization map $\Omega^{*} \to \text{MU}^{2*}$ which is an
isomorphism when applied to a point, and hence our results also hold
in the setting of complex cobordism.

We now formulate our main theorem. In calculating the cobordism classes $[\oM_{0,n}]$, we are forced to understand more general classes in the cobordism ring. Namely, for
$X$ a smooth variety over $k$, let $p \colon X \to \Spec k$ be
the projection. For $\alpha \in \Omega^{*}(X)$, let
$\int_{X}\alpha := p_{*}(\alpha)$ where $p_{*}$ is the pushforward map induced by
$p$, so in particular
\begin{equation*}
  [\oM_{0,n}] = \int_{\oM_{0,n}} 1 \in \L^{*}.
\end{equation*}
For each $1 \le i \le n$, let $L_{i}$ be the tautological line bundle on
$\oM_{0,n}$ corresponding to the $i^{\text{th}}$ marked
point, and let
$\psi_{i} = c_{1}(L_{i}) \in \Omega^{1}(\oM_{0,n})$ be its first Chern
class, called a \textit{psi-class}. For integers $d_{1},\dots,d_{n} \ge 0$, define
\begin{equation*}
  \psi^d = \psi_1^{d_1} \dotsm \psi_n^{d_n} ,
\end{equation*}
and for $1\le i\le n$, define
\begin{equation*}
  \psi^d/\psi_i =
  \begin{cases}
    \psi_1^{d_1} \dotsm \psi_i^{d_i-1} \dotsm \psi_n^{d_n} , & d_i>0 , \\
    0 , & d_i=0 .
  \end{cases}
\end{equation*}
Let $|d|=d_1+\cdots+d_n$. Our main theorem deals with calculating pushforwards of the form
$\int_{\oM_{0,n}} \psi^d$, which we call \textit{psi-class
  intersections}.

For $T$ a subset of
$\{1,2,\dots,n\}$ with $\vert{T}\vert \ge 2$, $\vert{T^c}\vert \ge 2$ and
$\vert{T \cap \{1,2,3\}}\vert \le 1$, let
\begin{equation*}
i_{T} \colon D^{T} \to \oM_{0,n}
\end{equation*}
be the inclusion of the divisor
$D^{T} = \oM_{0,\vert{T}\vert+1} \times \oM_{0,\vert{T^c}\vert+1}$ whose
generic points are curves with a single node, and the two components
containing markings labeled by the elements of $T$ and $T^{c}$. We denote by $\bullet$ the node-branch marking in
$\oM_{0,\vert{T}\vert+1}$ (or
$\oM_{0,\vert{T^c}\vert+1}$), and write
$L_{a} = \text{pr}_{1}^{*}L_{\bullet}$ and
$L_{b} = \text{pr}_{2}^{*}L_{\bullet}$, where
$\text{pr}_{1} \colon D^{T} \to \oM_{0,\vert{T}\vert+1}$ and
$\text{pr}_{2} \colon D^{T} \to \oM_{0,\vert{T^c}\vert+1}$ are the two
projections. We then define $\psi_{a} = c_{1}(L_{a})$ and
$\psi_{b} = c_{1}(L_{b})$. 

Let $x+_{\Omega}y \in \L^{*}[[x,y]]$ denote the universal
formal group law, and $\bar x$ its inverse, i.e. the unique power series in $\L^{*}[[x]]$ satisfying $x+_{\Omega}\bar x = 0$. We will be working with the power series
\begin{equation*}
\frac{1}{x(\bar x+_{\Omega}\bar y)}-\frac{1}{(\bar x +_{\Omega} \bar y)\bar y}+\frac{1}{\bar x \bar y} = \sum_{i,j \ge 0} b_{ij}x^{i}y^{j},
\end{equation*}
where explicitly, in low degrees,
\begin{equation*}
  b_{00} = [\P^{1}]^{2}-[\P^{2}], \hspace{0.5cm} b_{10} =
  [\Bl_{\text{pt}}\P^{3}]-[\P^{1}][\P^{2}], \hspace{0.5cm}
  b_{01} = 2([\Bl_{\text{pt}}\P^{3}]-[\P^{1}][\P^{2}]).
\end{equation*}
Here $\Bl_{\text{pt}}\P^{3}$ is the blow-up of $\P^{3}$ at a
point. Write
$x/\bar x = \sum_{j \ge 0} c_{j}x^{j}$, where in low degrees
$c_{0} = -1$, $c_{1} = [\P^{1}]$, and $c_{2} = 0$.

\begin{thm} \label{MainTheorem}
  The psi-class intersections are uniquely determined by the recursion
  \begin{multline*}
    \int_{\oM_{0,n+1}} 1\cdot\psi^d \\
    \begin{aligned}
    &= [\P^{1}] \int_{\oM_{0,n}} \psi^d
      - \sum_{i=1}^{n}\int_{\oM_{0,n}} \frac{\psi_{i}}{\bar\psi_{i}} \psi^d/\psi_i
      + \sum_{T}\int_{D^{T}} \biggl(
      \frac{1}{\psi_{a}\bigl(\bar\psi_{a}+_{\Omega}\bar\psi_{b}\bigr)}
     - \frac{1}{\bigl( \bar\psi_{a}+_{\Omega}\bar\psi_{b} \bigr) \bar\psi_{b}} + \frac{1}{\bar\psi_a\bar\psi_b}
      \biggr) i_{T}^{*} \psi^d \\
    &= [\P^{1}] \int_{\oM_{0,n}}  \psi^d -\sum_{i=1}^{n} \sum_{j \ge 0} \int_{\oM_{0,n}}
       c_{j}\psi_{i}^{j} \psi^d/\psi_i + \sum_{i,j \ge 0} \sum_{T}
       b_{ij} \int_{\oM_{0,\vert{T}\vert+1}} \psi_{\bullet}^{i} \prod_{\ell \in T}\psi_{\ell}^{d_\ell}
      \int_{\oM_{0,\vert{T^c}\vert+1}} \psi_{\bullet}^{j} \prod_{\ell \in T^{c}} \psi_{\ell}^{d_\ell},
    \end{aligned}
  \end{multline*}
  together with the initial condition $\int_{\oM_{0,3}} 1 = 1$.
\end{thm} 

Observe that if $|d| = 0$, then we get the equation
\begin{equation*}
[\oM_{0,n+1}] = [\P^{1}][\oM_{0,n}]+\sum_{i,j \ge 0}\sum_{T}b_{ij}\int_{\oM_{0,\vert{T}\vert+1}} \psi_{\bullet}^{i}\int_{\oM_{0,\vert{T^c}\vert+1}} \psi_{\bullet}^{j},
\end{equation*}
where $[\oM_{0,3}] = 1$. In particular, in order to calculate
the cobordism classes $[\oM_{0,n}] \in \L^{*}$, we are forced
to calculate the more general psi-class intersections.

Theorem \ref{MainTheorem} is a generalization to cobordism of the genus zero 
string equation, introduced by Witten \cite{Witten} for
rational cohomology (or the rational Chow groups) in arbitrary genus
$g$. On $\oM_{g,n}$, the psi-class intersections in
rational cohomology are governed by the string equation together with
the KdV hierarchy, known as Witten's conjecture, proved by
Kontsevich. However, in genus zero the psi-class intersections are
determined only by the string equation, which is implied by Theorem
\ref{MainTheorem} since algebraic (or complex) cobordism is the
universal (complex) oriented cohomology theory.

More precisely, we have an orientation map $\Omega^{*} \to \text{CH}^{*}$ to the
Chow ring. We also have an orientation map
$\Omega^{*} \to K^{0}[\beta,\beta^{-1}]$ to $K$-theory, where $\beta$ is a formal
variable of degree $-1$, and
$K^{0}(\Spec k)[\beta,\beta^{-1}] = \Z[\beta,\beta^{-1}]$. By mapping
Theorem \ref{MainTheorem} to the Chow ring and $K$-theory by
these orientation maps, we get known explicit closed formulas for
their respective genus zero psi-class intersections, due to work by
Givental and Lee on Quantum $K$-theory (see in particular \cite{Lee}
and \cite{Lee2}).

In the case of $K$-theory, we obtain a closed formula by defining
certain twists of the psi-classes, namely we define
$\Psi_{i} = [\Ll_{i}]\psi_{i}$, where $[\Ll_{i}]$ is the class of the
tautological line bundle $L_{i}$ in the $K$-theory of $\oM_{0,n}$.
\begin{exmp} \label{ClosedFormulas} The psi-class intersections in the
  Chow ring and $K$-theory are given by
  \begin{equation*}
    \int_{\oM_{0,n}} \psi^d =
    \binom{n-3}{d_{1},\dots,d_{n}} \quad \text{and} \quad
    \int_{\oM_{0,n}} \Psi^d =  \beta^{n-3-|d|}
    \binom{n-3}{n-3-|d|,d_{1},\dots,d_{n}},
  \end{equation*}
  which are the coefficients of $(t_{1}+\dots+t_{n})^{n-3}$ and $(\beta+t_{1}+\dots+t_{n})^{n-3}$ respectively, where the $t_{i}$ are formal variables. Observe that the formula for the $K$-theoretic
  intersections agrees with the formula for the Chow ring
  intersections when $|d| = n-3$. In fact, this is true more generally in
  cobordism.
\end{exmp}

\subsection{Calculations}
We can use Theorem \ref{MainTheorem} to explicitly calculate
cobordism-valued psi-class intersections in low dimensions after
fixing a choice of generators for the Lazard ring
$\L^{*} \cong \Z[u_{1},u_{2},u_{3},\dots]$. Our choice of generators
corresponds to the presentation of the universal formal group law
(Equation \ref{UnivFGL}) given in Appendix \ref{Appendix}. Explicitly, in
low degrees we have that
\begin{itemize}
\item[] $u_{1} = [\P^{1}]$
\item[] $u_{2} = [\P^{2}]$
\item[] $u_{3} = [\Bl_{\text{pt}}\P^{3}]-[\P^{1}]^{3}$
\item[] $u_{4} = [\P_{\P^{2}}(\Oo(1) \oplus \Oo \oplus \Oo)]-[\P^{1}]^{4}-3[\P^{1}]^{2}[\P^{2}]+4[\P^{1}][\Bl_{\text{pt}}\P^{3}]$.
\end{itemize}
Note that $[\Bl_{\text{pt}}\P^{3}]-[\P^{1}]^{3}$ is not equal to $[\P^{3}]$, as a calculation of the Chern numbers of these varieties shows. We discuss a method to obtain an expression for $u_{5}$ without denominators in Example \ref{LowDegGen} involving Milnor hypersurfaces. For brevity, we only mention that
\begin{equation*}
2u_{5} = [\P_{\P^{3}}(\Oo(1) \oplus \Oo \oplus \Oo)]-[\P^{1}][\P_{\P^{2}}(\Oo(1) \oplus \Oo \oplus \Oo)]-[\P^{1}]^{3}[\P^{2}]-[\P^{1}][\P^{2}]^{2}.
\end{equation*}

In terms of these generators, we obtain the following expressions for the cobordism classes $[\oM_{0,n}]$ for $n \le 8$. In Section \ref{Tables}, we provide a complete calculation of all psi-class intersections on $\oM_{0,n}$ for $n \le 8$ in terms of these generators.
\begin{exmp} \label{ExCalc}
We have that
\begin{itemize}
\item[] $[\oM_{0,3}] = 1$
\item[] $[\oM_{0,4}] = u_{1}$
\item[] $[\oM_{0,5}] = 4u_{1}^{2}-3u_{2}$
\item[] $[\oM_{0,6}] = 31u_{1}^{3}-30u_{1}u_{2}+17u_{3}$
\item[] $[\oM_{0,7}] = 273u_{1}^{4}-317u_{1}^{2}u_{2}+70u_{2}^{2}+214u_{1}u_{3}-25u_{4}$
\item[] $[\oM_{0,8}] = 2898u_{1}^{5}-4063u_{1}^{3}u_{2}+2012u_{1}u_{2}^{2}+2765u_{1}^{2}u_{3}-1204u_{2}u_{3}-385u_{1}u_{4}+461u_{5}$.
\end{itemize}
\end{exmp}
Coates and Givental \cite{CoatesGivental} discuss psi-class intersections in rational complex cobordism theory, which can be computed using the Hirzebruch-Riemann-Roch formula in $\text{MU}^{*} \otimes \Q$, together with cohomological intersection theory on $\oM_{0,n}$. Their answer, however, is not an explicit formula, and is expressed abstractly via a symplectic formalism. They do, however, explicitly calculate the rational cobordism classes of $\oM_{0,n}$ for $n \le 6$ using the basis of $\text{MU}^{*}(\text{pt}) \otimes \Q = \Q[[\C P^1],[\C P^2], [\C P^3],\dots]$ given by complex projective spaces. For example, they calculate that
\begin{equation*}
[\oM_{0,6}] = \frac{45}{2}[\C P^{1}]^{3}-30[\C P^{1}][\C P^{2}]+\frac{17}{2}[\C P^{3}],
\end{equation*}
which agrees with our calculation in Example \ref{ExCalc}. Indeed, we can express $[\P^{3}]$ in terms of our integral basis,
\begin{equation*}
[\P^{3}] = u_{1}^{3}+2u_{3},
\end{equation*}
by calculating the Chern numbers of all varieties involved. We include in Example \ref{ExmpCalcs} expansions of $[\oM_{0,n}]$ for $n \le 8$ in the rational basis of projective spaces. 
\begin{rem}
Givental in his paper \cite{Givental} remarks that there is no simple formula for the string equation in quantum cobordism theory. However, we derive such a string equation in the case of genus zero with the target being a point. 
\end{rem}

\subsection{Outline of the proof}

The main discrepancy that arises when generalizing the classical proof of the string equation from the Chow ring to algebraic cobordism, is that for $\pi \colon \oM_{0,n+1} \to \oM_{0,n}$ the universal curve that forgets the $(n+1)^{\text{st}}$ marking, we have that the pushforward $\pi_{*}(1)$ is zero in the Chow ring, but it is non-zero in algebraic cobordism. Here
\begin{equation*}
\pi_{*}(1) = [\oM_{0,n+1} \xrightarrow{\pi} \oM_{0,n}]
\end{equation*}
is the cobordism class of the universal curve, where the precise notion of the cobordism class of a map is discussed in Section \ref{GeomDescr}. To calculate this pushforward and express it in terms of the psi-classes, we use Keel's description \cite{Keel} of $\oM_{0,n}$ as an iterated blow-up, together with Levine-Pandharipande's formula \cite{LevinePandharipande}
\begin{equation*}
[\Bl_{Z}X \to X] = [X \xrightarrow{\text{id}} X]-[\P(N_{Z/X} \oplus \Oo_{Z}) \to X]+[\P_{\P(N_{Z/X})}(\Oo(1) \oplus \Oo) \to X]
\end{equation*}
for the cobordism class of the blow-up along a smooth subvariety $Z \hookrightarrow X$
with normal bundle $N_{Z/X}$. We also make use of a theorem of Quillen, presented as Theorem \ref{Vishik} below,
to calculate the cobordism classes of the projective bundles appearing
in the blow-up formula.

The key technical input in our proof is an explicit calculation of the
normal bundles of the blow-up centers in Keel's construction. Namely, Keel constructs $\oM_{0,n+1}$ as a
sequence of blow-ups
\begin{equation*}
B_{n-2} \xrightarrow{f_{n-2}} B_{n-3} \xrightarrow{f_{n-3}} \dotsm \xrightarrow{f_{3}} B_{2} \xrightarrow{f_{2}} B_{1}
\end{equation*}
along smooth codimension two centers, where $B_{1} = \oM_{0,n} \times \P^{1}$ and the center of the blow-up $B_{k+1} \xrightarrow{f_{k+1}} B_{k}$ is isomorphic to the union of all divisors $D^{T}$ of $\oM_{0,n}$ with $\vert{T}\vert = k+1$, which are disjoint. Then, our key lemma is the following.
\begin{lem} \label{KeyLemma}
For $1 \le k \le n-3$ and $\vert{T^c}\vert = k+1$, we have that
\begin{equation*}
N_{D^T/B_k} = (L_{a}^{\vee} \otimes L_{b}^{\vee}) \oplus L_{b}^{\vee}.
\end{equation*}
\end{lem}

\subsection{Organization}

The paper is organized as follows. In Section \ref{Background}, we
introduce the theory of
algebraic cobordism, in particular the geometric description due to
Levine and Pandharipande, including the blow-up formula and Quillen's
formula for the pushforward of a projective bundle. Section
\ref{CobClassSection} is devoted to proving Lemma \ref{KeyLemma}, stated as Lemma \ref{Conormals} in Section \ref{NormalBundles}, and
applying Levine-Pandharipande's blow-up formula for the universal
curve on $\oM_{0,n}$.

In Section \ref{StringSection} we complete the proof of Theorem
\ref{MainTheorem}, and in Section \ref{KTheoryChow} discuss the cases
of the Chow ring and $K$-theory, in particular the closed formulas of
Example \ref{ClosedFormulas}. Finally, in Appendix \ref{Appendix}, we
present the calculations with the universal formal group law needed to
produce the tables in Section \ref{Tables}, which give all
cobordism-valued psi-class intersections on $\oM_{0,n}$ for
$n \le 8$.

\subsection*{Acknowledgments}

We thank our advisor Ezra Getzler for his advice and assistance.

\section{Algebraic cobordism} \label{Background}

\subsection{Basic properties of algebraic cobordism}

Let $k$ be an algebraically closed field of characteristic 0, which we
fix throughout this paper. Let $\Sm_{k}$ be the category of smooth
quasi-projective schemes over $k$, and $\Rings^{*}$ be the
category of graded commutative rings with unit. Algebraic cobordism,
first constructed functorially by Levine-Morel \cite{AlgCob}, is a
functor $\Omega^{*} \colon \Sm_{k}^{\text{op}} \to \Rings^{*}$
satisfying a list of axioms analogous to those in Quillen's paper
\cite{Quillen} on complex cobordism. In particular, $\Omega^{*}$ is
equipped with pushforward maps
\begin{equation*}
f_{*} \colon \Omega^{*}(Y) \to \Omega^{*+d}(X)
\end{equation*}
for projective morphisms $f \colon Y \to X$ of relative codimension $d$, which are homomorphisms of graded $\Omega^{*}(X)$-modules. That is, we have $f_{*}(f^{*}(\alpha)\cdot \beta) = \alpha \cdot f_{*}(\beta)$ for $\alpha \in \Omega^{*}(X)$ and $\beta \in \Omega^{*}(Y)$, also known as the projection formula. Using the pushforward maps, we define the first Chern class of a line bundle $L \to X$ by $c_{1}(L) := s^{*}s_{*}(1)$, where $1 \in \Omega^{0}(X)$ is the unit and $s \colon X \to L$ is the zero section. The higher Chern classes of vector bundles are constructed via Grothendieck's method using the projective bundle formula. For line bundles $L,M \to X$, we have
\begin{equation} \label{FormalGroupLaw}
c_{1}(L \otimes M) = c_{1}(L)+_{\Omega} c_{1}(M),
\end{equation}
where $x+_{\Omega} y \in \Omega^{*}(\Spec k)[[x,y]]$ is the unique formal group
law associated to $\Omega^{*}$. There exists a unique power
series $\bar x \in \Omega^{*}(\Spec k)[[x]]$ satisfying
$x+_{\Omega}\bar x = 0$, called the inverse of the formal group law, where
for a line bundle $L \to X$ we have
\begin{equation*}
c_{1}(L^{\vee}) = \overline{c_{1}(L)}.
\end{equation*}
\begin{thm}[Levine-Morel \cite{AlgCob}]
  We have that $\Omega^{*}(\Spec k)$ is isomorphic to the Lazard ring
  $\L^{*} \cong \Z[u_{1},u_{2},u_{3},\dots]$ classifying one-dimensional,
  commutative formal group laws, and $x+_{\Omega}y$ is the universal formal
  group law. Here the generator $u_{i}$ has degree $-i$.
\end{thm}

\subsection{Geometric description of algebraic cobordism} \label{GeomDescr}

We now give an explicit description of algebraic cobordism by generators and relations due to Levine-Pandharipande \cite{LevinePandharipande}, analogous to the generators and relations of complex cobordism \cite{Quillen}. 
\begin{defn}
Let $f \colon Y \to \P^{1}$ be a projective morphism with $Y$ of pure dimension such that $f^{-1}(0) = A \cup B$, where $A$ and $B$ are smooth Cartier divisors of $Y$ intersecting transversely. Then, we say that $f$ is a \textit{double point degeneration} over $0 \in \P^{1}$ with \textit{degeneracy locus} $D = A \cap B$. 
\end{defn}
In this situation, we have that $\Oo_{Y}(A+B) \vert_{D} = \Oo_{D}$. But $\Oo_{Y}(B) \vert_{D} = N_{D/A}$ and $\Oo_{Y}(A) \vert_{D} = N_{D/B}$ are the normal bundles of $D$ in $A$ and $B$ respectively, so 
\begin{equation*}
N_{D/A} \otimes N_{D/B} = \Oo_{D}.
\end{equation*}
This implies that the projective bundles $\P_{D}(N_{D/A} \oplus \Oo_{D})$ and $\P_{D}(N_{D/B} \oplus \Oo_{D})$ are isomorphic, and we let $\P(f)$ denote the total space of either of these two bundles. 

For a smooth quasi-projective scheme $X$, let $\M^{n}(X)$ denote the group completion of the monoid under disjoint union of isomorphism classes of projective morphisms $Y \to X$ of codimension $n$ with $Y$ smooth, and let $\M^{*}(X) = \bigoplus_{n \ge 0} \M^{n}(X)$. If we have a projective morphism $\pi \colon Y \to X \times \P^{1}$ such that $f = \text{pr}_{2} \circ \pi \colon Y \to \P^{1}$ is a double point degeneration over $0 \in \P^{1}$, where $\text{pr}_{2}$ is the projection onto the second component, and $\xi \in \P^{1}$ is a regular value of $f$, then we can consider the element
\begin{equation*}
r_{\pi,\xi} := [Y_{\xi} \to X]-[A \to X]-[B \to X]+[\P(f) \to X] \in \M^{*}(X),
\end{equation*}
where $Y_{\xi} = f^{-1}(\xi)$, which we call a \textit{double point relation}. Here the maps are given by composing the inclusions in $Y$ with $\text{pr}_{1} \circ \pi$, where $\text{pr}_{1} \colon X \times \P^1 \to X$ is the projection onto the first component. Then, let $\R(X) \subset \M^{*}(X)$ be the subgroup generated by the double point relations $r_{\pi,\xi}$ for any $\pi$ and $\xi$ (noting that $Y$ can also vary). 
\begin{thm}[Levine-Pandharipande \cite{LevinePandharipande}]
For $X \in \Sm_{k}$, there is a canonical isomorphism
\begin{equation*}
\Omega^{*}(X) \cong \M^{*}(X)/\R(X).
\end{equation*}
\end{thm}
From now on, we will simply identify $\Omega^{*}(X)$ with the geometric presentation $\M^{*}(X)/\R(X)$. We then have the following descriptions for the structure maps and products in algebraic cobordism. 
\begin{itemize}
\item[--] \textit{Pushforward}: If $f \colon X \to Z$ is a projective
  morphism, then the pushforward $f_{*}([Y \to X])$ equals
  $[Y \to Z]$, where $Y \to Z$ is given by composing with $f$.
\item[--] \textit{Pullback}: If $f \colon X \to Z$ is a smooth morphism,
  then the pullback $f^{*}([Y \to Z])$ equals
  $[X \times_{Z} Y \to X]$, where $X \times_{Z} Y \to X$ is induced by projection
  onto the first component.
\item[--] \textit{Unit}: The unit $1 \in \Omega^{0}(X)$ is given by the class $[X \xrightarrow{\text{id}} X]$ of the identity map.
\item[--] \textit{External product}: The external product
  $\alpha\times\beta$ of the two classes
  $\alpha = [Y \to X] \in \Omega^{*}(X)$ and
  $\beta = [Y' \to X'] \in \Omega^{*}(X')$ is the product
  $[Y \times_{k} Y' \to X \times_{k} X'] \in \Omega^{*}(X \times_{k} X')$.
\item[--] \textit{Product}: The product $\alpha\cdot\beta$ of the two classes
  $\alpha = [Y \to X] \in \Omega^{*}(X)$ and
  $\beta = [Z \to X] \in \Omega^{*}(X)$ equals
  $\delta^{*}(\alpha \times \beta)$, where
  $\delta \colon X \to X \times_{k} X$ is the diagonal embedding.
\end{itemize}
\begin{rem}
We will use the notation $[Y] := [Y \to \Spec k]$ for elements of $\Omega^{*}(\Spec k) = \L^{*}$. 
\end{rem}
An important property of algebraic cobordism which we use throughout this paper is that $\Omega^{i}(X) = 0$ for $i > \text{dim}\,X$. Another important property is that the class $[X] \in \L^{*}$ of a smooth projective variety $X$ is completely determined by its Chern numbers. 

For integers $m,n > 0$, define the line bundle $\Oo(1,1) = \text{pr}_{1}^{*}\Oo_{\P^m}(1) \otimes \text{pr}_{2}^{*}\Oo_{\P^n}(1)$ on $\P^{m} \times \P^{n}$. Then, the Milnor hypersurface $H_{m,n} \subset \P^{m} \times \P^{n}$ is the smooth closed subscheme defined by the zero locus of a section of $\Oo(1,1)$ that is transverse to the zero section. The following is proved by Levine-Morel \cite[Lemma 2.5.5]{AlgCob}, and is a straightforward application of Equation \ref{FormalGroupLaw}. 
\begin{lem} \label{Milnor}
Expand the universal formal group law as $x+_{\Omega}y = x+y+\sum_{i,j \ge 1} a_{ij}x^{i}y^{j}$. Then, we have that
\begin{equation*}
[H_{m,n}] = [\P^{m}][\P^{n-1}]+[\P^{m-1}][\P^{n}]+\sum_{i=1}^{m}\sum_{j=1}^{n}a_{ij}[\P^{m-i}][\P^{n-j}].
\end{equation*}
\end{lem}
\begin{exmp} \label{MilnorExmp}
Up to a change of coordinates, we have that $H_{1,1}$ is the diagonal copy of $\P^{1}$ inside $\P^{1} \times \P^{1}$. By Lemma \ref{Milnor}, we have
\begin{equation*}
[H_{1,1}] = [\P^{1}]+[\P^{1}]+a_{11},
\end{equation*}
which implies that $a_{11} = -[\P^{1}]$. 
\end{exmp}

\subsection{Pushforwards of blow-ups and projective bundles}

Let $i \colon Z \to X$ be a regular embedding in $\Sm_{k}$ with normal bundle $N_{Z/X}$, and let $\Bl_{Z}X$ denote the blow-up of $X$ along $Z$. Let $p \colon \P(N_{Z/X}) \to Z$ be the projection. We have the following formula for the cobordism class of the blow-up map \cite[Section 4.1]{LYZR}.
\begin{lem}[Blow-up formula] \label{Blowup}
We have the relation
\begin{align*}
[\Bl_{Z}X \to X] &= [X \xrightarrow{id} X]-i_{*}[\P(N_{Z/X} \oplus \Oo_{Z}) \to Z]+i_{*}p_{*}[\P_{\P(N_{Z/X})}(\Oo(1)\oplus\Oo) \to \P(N_{Z/X})] \\
&= [X \xrightarrow{id} X]-[\P(N_{Z/X} \oplus \Oo_{Z}) \to X]+[\P_{\P(N_{Z/X})}(\Oo(1)\oplus\Oo) \to X].
\end{align*}
\end{lem}
\begin{proof}
Consider the deformation to the normal cone, namely the blow-up of $X \times \P^{1}$ along $Z \times \{0\}$, and denote this blow-up by $Y$. Then, the blow-up map $Y \to X \times \P^{1}$ is a double-point degeneration over $0 \in \P^{1}$ with $A = \Bl_{Z}X$, $B = \P(N_{Z/X} \oplus \Oo_{Z})$, and degeneracy locus $D = \P(N_{Z/X})$, which has normal bundle $\Oo(1)$ inside $B$. Since the fiber over a regular value of $\P^{1}$ is equal to $X$, the blow-up formula is precisely the associated double point relation.
\end{proof}
We have a formula, due to Quillen, that can be used to calculate the projective bundle terms appearing on the right-hand side of the blow-up formula. See \cite[Theorem 5.30]{Vishik} for a proof of this result, and \cite[Proposition 2.1]{SchubertCalculus} for a geometric proof in the rank two case using the double point relation. 
\begin{thm}[Quillen] \label{Vishik}
Let $X \in \Sm_{k}$ and $V \to X$ be a vector bundle of rank $n$, with $\Omega^{*}$-Chern roots $\lambda_{1},\dots,\lambda_{n}$. Let $\pi \colon \P(V) \to X$ be the associated projective bundle, and fix a power series $f(t) \in \Omega^{*}(X)[[t]]$. Letting $\xi = c_{1}(\Oo_{\P(V)}(1))$, we have that 
\begin{equation*}
\pi_{\ast}(f(\xi)) = \sum_{i=1}^{n} \frac{f(\bar\lambda_{i})}{\prod_{j \neq i} (\lambda_{j}+_{\Omega}\bar\lambda_{i})} \in \Omega^{*}(X).
\end{equation*}
\end{thm}
\begin{rem}
The right-hand side of the formula in Theorem \ref{Vishik} is symmetric in the variables $\lambda_{1},\dots,\lambda_{n}$, and hence can be expressed as a power series in $c_{1}(V),\dots,c_{n}(V)$. This power series is in fact a polynomial, yielding a well-defined element of $\Omega^{*}(X)$, since $\Omega^{i}(X) = 0$ for $i > \text{dim}\,X$. 
\end{rem}
\begin{exmp}[Low degree generators of $\L^{*}$] \label{LowDegGen}
We will use the blow-up and projective bundle formulas, Lemma \ref{Blowup} and Theorem \ref{Vishik}, to determine explicit representatives for low degree generators $u_{i}$ of the Lazard ring $\L^{*}$, with these generators corresponding to the presentation of the universal formal group law $x+_{\Omega}y$ given in Appendix \ref{Appendix}, Equation \ref{UnivFGL}. Let $\text{pt} = \Spec k$. First, we apply Theorem \ref{Vishik} to the trivial bundle $\Oo_{\text{pt}} \oplus \Oo_{\text{pt}}$, so $\lambda_{1} = \lambda_{2} = 0$. As discussed in Appendix \ref{Appendix} (namely Equation \ref{varphi}), the constant term of the power series
\begin{equation*}
\frac{1}{y+_{\Omega}\bar x}+\frac{1}{x+_{\Omega}\bar y}
\end{equation*}
is $u_{1}$. Hence, it follows that
$u_{1} = [\P(\Oo_{\text{pt}} \oplus \Oo_{\text{pt}})] = [\P^{1}]$, which
can alternatively be deduced by Example~\ref{MilnorExmp}. Similarly,
the constant term of the series
\begin{equation}
  \label{3term}
  \frac{1}{(y+_{\Omega}\bar x)(z+_{\Omega}\bar x)} + \frac{1}{(x+_{\Omega}\bar y)
    (z+_{\Omega}\bar y)}+\frac{1}{(x+_{\Omega}\bar z)(y+_{\Omega}\bar z)}
\end{equation}
is $u_{2}$, by Remark \ref{AppendixRemark}, which implies that
$u_{2} = [\P(\Oo_{\text{pt}} \oplus \Oo_{\text{pt}} \oplus \Oo_{\text{pt}})] =
[\P^{2}]$. However, the generator $u_{3}$ is more involved. Applying
Lemma~\ref{Blowup} and pushing forward to a point in the case
$X = \P^{3}$ and $Z = \text{pt}$, we obtain
\begin{align*}
[\Bl_{\text{pt}}\P^{3}] &= [\P^{3}]-[\P^{3}]+[\P_{\P^2}(\Oo(1) \oplus \Oo)] \\
&= [\P_{\P^2}(\Oo(1) \oplus \Oo)].
\end{align*}
\end{exmp}
We calculate the class of this projective bundle again using Theorem \ref{Vishik}, where $\lambda_{1} = 0$ and $\lambda_{2} = [\P^{1} \to \P^{2}]$, noting that $\lambda_{2}^{2} = [\text{pt} \to \P^{2}]$. By Equation \ref{varphi} in Appendix \ref{Appendix}, up to degree 2, we have that
\begin{equation*}
\frac{1}{x}+\frac{1}{\bar x} = u_{1}+(u_{1}^{3}-u_{1}u_{2}+u_{3})x^{2},
\end{equation*}
and so, using that $u_{2} = [\P^{2}]$,
\begin{equation*}
[\Bl_{\text{pt}}\P^{3}] = u_{1}u_{2}+u_{1}^{3}-u_{1}u_{2}+u_{3} = u_{1}^{3}+u_{3}.
\end{equation*}
That is, we have $u_{3} = [\Bl_{\text{pt}}\P^{3}]-[\P^{1}]^{3}$. A different method is required to calculate a representative for $u_{4}$, since the coefficient of $x^{3}$ in $1/x+1/\bar x$ does not contain $u_{4}$ as a term. However, by substituting $y = z = 0$ into the power series \ref{3term}, we have (again by Remark \ref{AppendixRemark}), up to degree 2, the polynomial
\begin{equation*}
u_{2}+(-3u_{1}^{4}+3u_{1}^{2}u_{2}-u_{2}^{2}-4u_{1}u_{3}+u_{4})x^{2}.
\end{equation*}
Hence, we can take 
\begin{align*}
u_{4} &= [\P_{\P^{2}}(\Oo(1) \oplus \Oo \oplus \Oo)]+3u_{1}^{4}-3u_{1}^{2}u_{2}+4u_{1}u_{3} \\
&= [\P_{\P^{2}}(\Oo(1) \oplus \Oo \oplus \Oo)]-[\P^{1}]^{4}-3[\P^{1}]^{2}[\P^{2}]+4[\P^{1}][\Bl_{\text{pt}}\P^{3}].
\end{align*}
By the same method using Remark \ref{AppendixRemark}, we get that
\begin{equation} \label{2u5}
2u_{5} = [\P_{\P^{3}}(\Oo(1) \oplus \Oo \oplus \Oo)]-[\P^{1}][\P_{\P^{2}}(\Oo(1) \oplus \Oo \oplus \Oo)]-[\P^{1}]^{3}[\P^{2}]-[\P^{1}][\P^{2}]^{2}.
\end{equation}
A more complicated expression for $u_{5}$ without denominators in terms of the Milnor hypersurfaces $H_{3,3}$, $H_{5,1}$ and $H_{4,2}$ can be obtained by Lemma \ref{Milnor}, since by Equation \ref{UnivFGL} the coefficient of $u_{5}$ in $a_{33}-a_{51}-a_{42}$ is equal to 1, where $a_{ij}$ is the coefficient of $x^{i}y^{j}$.

\section{The cobordism class of $\oM_{0,n}$} \label{CobClassSection}

For $n \ge 3$, let $\oM_{0,n}$ be the moduli space of stable curves of genus zero with $n$ marked points. Note that $\oM_{0,n}$ is a proper smooth variety of dimension $n-3$, and hence defines an element $[\oM_{0,n}] \in \L^{-(n-3)}$. In this section, we derive the technical input for our inductive formula for the cobordism class $[\oM_{0,n}]$, which will be complete in Section \ref{StringSection}. 

\subsection{Keel's blow-up description of $\oM_{0,n}$} 

We now recall Keel's construction \cite{Keel} of $\oM_{0,n}$ as an iterated blow-up along smooth centers. Consider the diagram
\begin{equation*}
  \begin{tikzcd}[column sep=6em]
    \oM_{0,n+1} \ar[dr,"\pi"'] \ar[r,"{(\pi,\pi_{1,2,3,n+1})}"] 
    & \oM_{0,n} \times \P^{1}
    \ar[d,"p_1"] \\
    & \oM_{0,n}. \ar[ul,bend left=20,"\sigma^i"]
  \end{tikzcd}
\end{equation*}
Here $\pi$ forgets the $(n+1)^{\text{st}}$-marked point, $\sigma^{i}$ is the section corresponding to the $i^{\text{th}}$-marked point, $p_{1}$ is the projection onto the first component, and $\pi_{1,2,3,n+1}$ forgets all markings except those labeled by $1$, $2$, $3$ and $n+1$, where we identify $\oM_{0,4}$ with $\P^{1}$. For $T \subset \{1,2,3,\dots,n\}$ with $\vert{T}\vert \ge 2$, $\vert{T^{c}}\vert \ge 2$ and $\vert{T \cap \{1,2,3\}}\vert \le 1$, let $D^{T}$ be the divisor of $\oM_{0,n}$ with generic element a curve having a single node, and the two components containing markings labeled by the elements of $T$ and $T^{c}$ respectively, so that
\begin{equation*}
D^{T} \cong \oM_{0,\vert{T}\vert+1} \times \oM_{0,\vert{T^{c}}\vert+1}.
\end{equation*}
Since $D^{T} = D^{T^c}$, the condition $\vert{T \cap \{1,2,3\}}\vert \le 1$ ensures that we do not overcount divisors. Denote the map $(\pi,\pi_{1,2,3,n+1})$ by $\pi_{1}$. We embed $D^{T}$ inside $\oM_{0,n} \times \P^{1}$ as 
\begin{equation*}
S_{1}^{T} := (\pi_{1} \circ \sigma^{i})(D^{T})
\end{equation*}
for $i \in T$, which is independent of the choice of $i \in T$. Let $B_{1} = \oM_{0,n} \times \P^{1}$, and let $B_{2}$ be the blow-up of $B_{1}$ along the union of all $S_{1}^{T}$ such that $\vert{T^{c}}\vert = 2$, which are disjoint, so the center of this blow-up is a smooth codimension two subvariety. Following this process, we inductively define a sequence of blow-ups
\begin{equation*}
  \begin{tikzcd}
    B_{k+1} \ar[rrrr,bend right=26,"g_{k+1}"']
    \ar[r,"f_{k+1}"] & B_{k} \ar[r,"f_{k}"] & \dotsm
    \ar[r,"f_{3}"] & B_{2} \ar[r,"f_{2}"] & B_{1},
  \end{tikzcd}
\end{equation*}
where $B_{k+1}$ is the blow-up of $B_{k}$ along the union of the strict transforms (under the composition $B_{k} \xrightarrow{g_{k}} B_{1}$) of all the $S_{1}^{T}$ such that $\vert{T^{c}}\vert = k+1$, which are again disjoint and thus form a smooth codimension two center. We denote these strict transforms by $S_{k}^{T} \subset B_{k}$. 
\begin{thm}[Keel \cite{Keel}] \label{Keel}
For every $k$, the map $\pi_{1}$ factors through $B_{k}$
\begin{equation*}
  \begin{tikzcd}[column sep=5em,row sep=3em]
    \oM_{0,n+1} \ar[dr,"\pi_{1}"'] \ar[r,"\pi_{k}"] & B_{k} \ar[d,"g_{k}"'] \\
    & B_{1}, \ar[ul,bend left=25,"\sigma^{i} \circ p_{1}"] \ar[u,bend
    right=25,"\sigma_{k}^{i}"']
  \end{tikzcd}
\end{equation*}
inducing sections $\sigma_{k}^{1},\dots,\sigma_{k}^{n}$ of $B_{k} \xrightarrow{g_{k}} B_{1}$, with $S_{k}^{T} = \sigma_{k}^{i}(S_{1}^{T})$ for any $i \in T$, and $\pi_{n-2} \colon \oM_{0,n+1} \to B_{n-2}$ is an isomorphism. 
\end{thm}

\subsection{Normal bundles of Keel's blow-up centers} \label{NormalBundles}

In this section, we derive a formula for the normal bundle $N_{S_{k}^{T}/B_{k}}$ of the blow-up center $S_{k}^{T}$ inside $B_{k}$, where $\vert{T}\vert = k+1$, in terms of tautological line bundles. 

For each $k \ge 2$, let $P_{k}^{T}$ be the strict transform of $P_{1}^{T} := D^{T} \times \P^{1} \subset B_{1}$ under the composition of blow-ups $g_{k}$. For $1 \le i \le n$, let $\Sigma_{k}^{i}$ be the strict transform of $\Sigma_{1}^{i} := (\pi_{1} \circ \sigma^{i})(\oM_{0,n}) \subset B_{1}$ under $g_k$.
\begin{lem}[Keel \cite{Keel}] \label{CompleteIntersection}
For every $k \ge 1$ and $i \in T$, we have that $S_{k}^{T}$ is the complete intersection of the divisors $P_{k}^{T}$ and $\Sigma_{k}^{i}$ of $B_{k}$. 
\end{lem}
For each $1 \le i \le n$, let $L_{i} = (\sigma^{i})^{*}\omega_{\oM_{0,n+1}/\oM_{0,n}}$ denote the $i^{\text{th}}$ tautological line bundle on $\oM_{0,n}$, where $\omega_{\oM_{0,n+1}/\oM_{0,n}}$ is the relative dualizing sheaf. In particular, the fiber of $L_{i}$ over a stable curve $C$, with marked points $x_{1},\dots,x_{n}$, is the cotangent space $T_{x_{i}}^{*}C$. Recall that on $\oM_{0,n+1}$, we have the relation
\begin{equation} \label{ForgetfulRelation}
L_{i} = \pi^{*}L_{i} \otimes \Oo_{\oM_{0,n+1}}(D^{\{i,n+1\}}).
\end{equation}
By applying Equation \ref{ForgetfulRelation} inductively, using that $\oM_{0,3}$ is a point, we obtain the following well-known description of the tautological line bundles on $\oM_{0,n}$.
\begin{lem}
  \label{TautDescription}
  For fixed distinct $1 \le i,j,k \le n$, we have that
  \begin{equation*}
    L_{i} = \Oo_{\oM_{0,n}} \biggl( \sum_{\substack{i \in T \\ j,k \in T^{c}}}
      D^{T} \biggr) .
  \end{equation*}
\end{lem}
Identifying $D^{T}$ with $\oM_{0,\vert{T}\vert+1} \times \oM_{0,\vert{T^{c}}\vert+1}$, let $\text{pr}_{1} \colon D^{T} \to \oM_{0,\vert{T}\vert+1}$ and $\text{pr}_{2} \colon D^{T} \to \oM_{0,\vert{T^{c}}\vert+1}$ be the two projections, and denote by $L_{\bullet}$ the tautological line bundle on $\oM_{0,\vert{T}\vert+1}$ (or $\oM_{0,\vert{T^{c}}\vert+1}$) corresponding to the branch point. Then, letting $L_{a} = \text{pr}_{1}^{*}L_{\bullet}$ and $L_{b} = \text{pr}_{2}^{*}L_{\bullet}$, it is a well-known fact that
\begin{equation*}
N_{D^{T}/\oM_{0,n}} = L_{a}^{\vee} \otimes L_{b}^{\vee}.
\end{equation*}
\begin{lem} \label{Conormals}
Let $1 \le k \le n-3$ and $\vert{T^{c}}\vert = k+1$. Under the isomorphism $S_{k}^{T} \cong D^{T}$ given by the section $\sigma_{k}^{i}$ for any $i \in T$, we have that
\begin{equation*}
N_{D^{T}/B_{k}} = (L_{a}^{\vee} \otimes L_{b}^{\vee}) \oplus L_{b}^{\vee}.
\end{equation*}
\end{lem}
\begin{rem}
More precisely, by identifying $S_{1}^{T}$ with $D^{T}$ via the section $\pi_{1} \circ \sigma^{i}$ for any $i \in T$, we have that
\begin{equation*}
(\sigma_{k}^{i} \vert_{D^{T}})^{*}N_{S_{k}^{T}/B_{k}} = (L_{a}^{\vee} \otimes L_{b}^{\vee}) \oplus L_{b}^{\vee},
\end{equation*}
noting that $\sigma_{k}^{i} \vert_{D^{T}}$ is an isomorphism. We will generally suppress the pullback by this isomorphism and write $D^{T}$ instead of $S_{k}^{T}$ for ease of notation. 
\end{rem}
\begin{proof}
By Lemma \ref{CompleteIntersection}, it follows that
\begin{equation*}
N_{S_{k}^{T}/B_{k}} = \Oo_{S_{k}^T}(P_{k}^{T}) \oplus \Oo_{S_{k}^T}(\Sigma_{k}^{i})
\end{equation*}
for $i \in T$. Since $\vert{T^{c}}\vert \neq k$ and $\Sigma_{k-1}^{i}$ contains $S_{k-1}^{V}$ for every $V$ containing $i$ with $\vert{V^c}\vert = k$, we know (see \cite[Lemma 4.3]{Keel} and the succeeding Claim in \cite{Keel}) that
\begin{equation*}
P_{k}^{T} = f_{k}^{*}P_{k-1}^{T} \hspace{0.5cm} \text{ and } \hspace{0.5cm} \Sigma_{k}^{i} = f_{k}^{*}\Sigma_{k-1}^{i}-\sum_{\substack{i \in V \\ \vert{V^{c}}\vert = k}} E_{k}^{V},
\end{equation*}
where the exceptional divisor $E_{k}^{V} =
f_{k}^{-1}(S_{k-1}^{V})$. Hence, it follows that
\begin{align*}
N_{S_{k}^{T}/B_{k}} &= \Oo_{S_{k}^{T}}(f_{k}^{*}P_{k-1}^{T}) \oplus
  \Oo_{S_{k}^{T}} \biggl(f_{k}^{*}\Sigma_{k-1}^{i}-\sum_{\substack{i \in V \\
  \vert{V^{c}}\vert = k}} E_{k}^{V} \biggr) \\
&= \Oo_{S_{k}^{T}}(f_{k}^{*}P_{k-1}^{T}) \oplus \Oo_{S_{k}^{T}} \biggl(
  f_{k}^{*}\Sigma_{k-1}^{i}-\sum_{\substack{i \in V \\ \vert{V^{c}}\vert = k}} S_{k}^{T}
  \cap E_{k}^{V} \biggr).
\end{align*}
However, note that each $S_{k}^{T} \cap E_{k}^{V}$ is precisely the exceptional divisor of the blow-up of $S_{k-1}^{T}$ along the intersection $S_{k-1}^{T} \cap S_{k-1}^{V}$ (see \cite[Proposition IV-21]{EisenbudHarris}). But we know by \cite[Lemma 4.4]{Keel} that $S_{k-1}^{T}$ and $S_{k-1}^{V}$ have non-empty intersection if and only if $T$ is a subset of $V$, in which case they intersect in a smooth Cartier divisor of each. So, if $T \not\subset V$, then $S_{k}^{T} \cap E_{k}^{V} = \emptyset$, and otherwise, we have that 
\begin{equation*}
S_{k}^{T} \cap E_{k}^{V} = (f_{k} \vert_{S_{k}^{T}})^{*}(S_{k-1}^{T} \cap S_{k-1}^{V})
\end{equation*}
as divisors of $S_{k}^{T}$, and $f_{k} \vert_{S_{k}^{T}} \colon S_{k}^{T} \to S_{k-1}^{T}$ is an isomorphism. Hence, it follows that
\begin{equation*}
N_{S_{k}^{T}/B_{k}} = (f_{k} \vert_{S_{k}^T})^{*}\Oo_{S_{k-1}^{T}}(P_{k-1}^{T}) \oplus (f_{k} \vert_{S_k^{T}})^{*} \Oo_{S_{k-1}^{T}} \biggl(\Sigma_{k-1}^{i}-\sum_{\substack{V \supset T \\ \vert{V^{c}}\vert = k}} S_{k-1}^{T} \cap S_{k-1}^{V}\biggr).
\end{equation*}
Repeating the above procedure, and using that $S_{k-1}^{T}$ and $S_{k-1}^{V}$ for $V \supset T$ and $\vert{V^c}\vert = k$ intersect the centers of the blow-up $f_{k-1}$ transversely in Cartier divisors (if the intersection is non-empty) so that
\begin{equation*}
S_{k-1}^{T} \cap S_{k-1}^{V} = (f_{k-1} \vert_{S_{k-1}^{T}})^{*}(S_{k-2}^{T} \cap S_{k-2}^{V}),
\end{equation*}
i.e. the intersection of their strict transforms is the strict transform of their intersection, we get the equality
\begin{equation*}
N_{S_{k}^{T}/B_{k}} = (f_{k-1} \circ f_{k} \vert_{S_{k}^T})^{*}\Oo_{S_{k-2}^{T}}(P_{k-2}^{T}) \oplus (f_{k-1} \circ f_{k} \vert_{S_k^{T}})^{*} \Oo_{S_{k-2}^{T}} \biggl(\Sigma_{k-2}^{i}-\sum_{\substack{V \supset T \\ k-1 \le \vert{V^{c}}\vert \le k}} S_{k-2}^{T} \cap S_{k-2}^{V}\biggr).
\end{equation*}
Continuing this process, it follows that
\begin{equation*}
N_{S_{k}^{T}/B_{k}} = (g_{k} \vert_{S_{k}^{T}})^{*} \Oo_{S_{1}^{T}}(P_{1}^{T}) \oplus (g_{k} \vert_{S_{k}^{T}})^{*} \Oo_{S_{1}^{T}}\biggl(\Sigma_{1}^{i}-\sum_{V \supset T} S_{1}^{T} \cap S_{1}^{V}\biggr).
\end{equation*}
The normal bundle of $P_{1}^{T} = D^{T} \times \P^{1}$ inside $B_{1} = \oM_{0,n} \times \P^{1}$ is $\Oo_{P_{1}^{T}}(P_{1}^{T}) = \text{pr}_{1}^{*} (L_{a}^{\vee}\otimes L_{b}^{\vee})$. Restricting to $S_{1}^{T}$, which we identify with $D^{T}$, we have that $\Oo_{S_{1}^{T}}(P_{1}^{T}) = L_{a}^{\vee}\otimes L_{b}^{\vee}$. We now analyze the second summand. 

\paragraph{Case 1.} Assume that $\vert{T \cap \{1,2,3\}}\vert = 1$. Observe that $S_{1}^{T} = (\pi_{1} \circ \sigma^{i})(D^{T}) = D^{T} \times \text{pt}$, so $\Oo_{S_{1}^{T}}(\Sigma_{1}^{i}) = \Oo_{S_{1}^{T}}$. Identifying $S_{1}^{T}$ and $S_{1}^{V}$ with $D^{T}$ and $D^{V}$ respectively, we have
\begin{equation*}
\Oo_{S_{1}^{T}}\biggl(\Sigma_{1}^{i}-\sum_{V \supset T} S_{1}^{T} \cap S_{1}^{V}\biggr) = \Oo_{D^T}\biggl(-\sum_{V \supset T} D^{T} \cap D^{V}\biggr).
\end{equation*}
We can write $D^{T} \cap D^{V} = \oM_{0,\vert{T}\vert+1} \times D_{\vert{T^{c}}\vert+1}^{(V \setminus T) \cup \{\bullet\}}$, where $D_{\vert{T^{c}}\vert+1}^{(V \setminus T) \cup \{\bullet\}}$ denotes the corresponding divisor $D^{(V \setminus T) \cup \{\bullet\}}$ of $\oM_{0,\vert{T^{c}}\vert+1}$, and $\bullet$ is the branch point. Here we view $\oM_{0,\vert{T^{c}}\vert+1}$ as having marked points labeled by the elements of $T^{c} \cup \{\bullet\}$, and note that $((V \setminus T) \cup \{\bullet\})^{c} = V^{c}$. Assume, without loss of generality, that $1 \in T$. Then, since $T \subset V$ and $\vert{V \cap \{1,2,3\}}\vert \le 1$, we have that $2,3 \in V^{c}$, and so by re-labeling we get
\begin{equation*}
\Oo_{\oM_{0,\vert{T^{c}}\vert+1}}\biggl(\sum_{V \supset T} D_{\vert{T^{c}}\vert+1}^{(V \setminus T) \cup \{\bullet\}}\biggr) = \Oo_{\oM_{0,\vert{T^{c}}\vert+1}}\biggl(\sum_{\substack{\bullet \in W \\ 2,3 \in W^{c}}} D_{\vert{T^{c}}\vert+1}^{W}\biggr) = L_{\bullet}
\end{equation*}
by Lemma \ref{TautDescription}. Hence, it follows that
\begin{align*}
\Oo_{D^T}\biggl(-\sum_{V \supset T} D^{T} \cap D^{V}\biggr) &= \Oo_{\oM_{0,\vert{T}\vert+1} \times \oM_{0,\vert{T^{c}}\vert+1}}\biggl(\oM_{0,\vert{T}\vert+1} \times \biggl(-\sum_{\substack{\bullet \in W \\ 2,3 \in W^{c}}} D_{\vert{T^{c}}\vert+1}^{W}\biggr)\biggr) \\
&= \text{pr}_{2}^{*} L_{\bullet}^{\vee} = L_{b}^{\vee}.
\end{align*}

\paragraph{Case 2.} Now assume that $\vert{T \cap \{1,2,3\}}\vert = 0$. Under the identifications of $S_{1}^{T}$ and $\Sigma_{1}^{i}$ with $D^{T}$ and $\oM_{0,n}$, respectively, we have that $\Oo_{S_{1}^{T}}(\Sigma_{1}^{i})$ is the restriction to $D^{T}$ of the normal bundle to the graph of 
\begin{equation*}
\pi_{1,2,3,n+1} \circ \sigma^{i} \colon \oM_{0,n} \to \oM_{0,4} = \P^{1},
\end{equation*}
i.e. we have
\begin{equation*}
\Oo_{S_{1}^{T}}(\Sigma_{1}^{i}) = (\pi_{1,2,3,n+1} \circ \sigma^{i} \vert_{D^{T}})^{*}T_{\P^{1}} = (\pi_{1,2,3,n+1} \circ \sigma^{i} \vert_{D^{T}})^{*}\Oo_{\P^{1}}(2).
\end{equation*}
Let $\pi_{1,2,3,i} \colon \oM_{0,n} \to \oM_{0,4}$ forget all but markings $1,2,3,i$. Then, note that
\begin{equation*}
\pi_{1,2,3,n+1} \circ \sigma^{i} \vert_{D^{T}} = \pi_{1,2,3,i} \vert_{D^{T}}
\end{equation*}
if we view elements of $\oM_{0,4}$ as having marked points labeled by the set $\{1,2,3,4\}$, and the two forgetful maps re-label the marked points $n+1$ and $i$ respectively with $4$. Indeed, both these maps are equal to $\pi_{1,2,3,\bullet} \circ \text{pr}_{2} \colon D^{T} \to \oM_{0,4}$, where the map $\pi_{1,2,3,\bullet} \colon \oM_{0,\vert{T^{c}}\vert+1} \to \oM_{0,4}$ is defined analogously (see \cite[Fact 0.2]{Keel}). Hence, we have that
\begin{equation*}
\Oo_{S_{1}^{T}}(\Sigma_{1}^{i}) = (\pi_{1,2,3,i} \vert_{D^T})^{*}\Oo_{\P^1}(2),
\end{equation*}
where for simplicity we now re-label the marked point $4$ with $i$ in $\oM_{0,4} = \P^{1}$. Note that the divisors $D^{\{1,i\}}$ and $D^{\{2,i\}}$ of $\oM_{0,4}$ are points in $\P^{1}$, hence are linearly equivalent, and so we can write
\begin{equation*}
\Oo_{S_{1}^{T}}(\Sigma_{1}^{i}) = (\pi_{1,2,3,i} \vert_{D^T})^{*}\Oo_{\oM_{0,4}}(D^{\{1,i\}}+D^{\{2,i\}}).
\end{equation*}
We have that
\begin{equation*}
\pi_{1,2,3,i}^{*}(D^{\{1,i\}}+D^{\{2,i\}}) = \sum_{\substack{1,i \in V \\ 2,3 \in V^{c}}} D^{V}+\sum_{\substack{2,i \in V \\ 1,3 \in V^{c}}} D^{V}.
\end{equation*}
Intersecting with $D^{T}$, we get that
\begin{equation*}
\Oo_{S_{1}^{T}}(\Sigma_{1}^{i}) = \Oo_{D^T}\biggl(\sum_{\substack{1,i \in V \\ 2,3 \in V^{c}}} D^{T} \cap D^{V}+\sum_{\substack{2,i \in V \\ 1,3 \in V^{c}}} D^{T} \cap D^{V}\biggr),
\end{equation*}
and hence
\begin{equation*}
  \Oo_{S_{1}^{T}}\biggl(\Sigma_{1}^{i}-\sum_{V \supset T} S_{1}^{T} \cap S_{1}^{V}\biggr)
  = \Oo_{D^T} \biggl( \sum_{\substack{1,i \in V \\ 2,3 \in V^{c}}} D^{T} \cap
  D^{V}+\sum_{\substack{2,i \in V \\ 1,3 \in V^{c}}} D^{T} \cap D^{V}-\sum_{V \supset T}
  D^{T} \cap D^{V}\biggr).
\end{equation*}
However, note that if $V \supset T$ is arbitrary, then since $\vert{V \cap \{1,2,3\}}\vert \le 1$, we know that $1 \in V$ implies that $2,3 \in V^{c}$. Similarly, if $2 \in V$, then $1,3 \in V^{c}$. Moreover, since $i \in T$, we know that $i \in V$. Hence, it follows that 
\begin{equation*}
  \sum_{\substack{1,i \in V \\ 2,3 \in V^{c}}} D^{T} \cap D^{V}+\sum_{\substack{2,i
      \in V \\ 1,3 \in V^{c}}} D^{T} \cap D^{V}-\sum_{V \supset T} D^{T} \cap D^{V} =
  -\sum_{\substack{V \supset T \\ 1,2 \in V^{c}}} D^{T} \cap D^{V},
\end{equation*}
and so analogous to the proof of Case 1, we get that
\begin{align*}
  \Oo_{S_{1}^{T}}\biggl(\Sigma_{1}^{i}-\sum_{V \supset T} S_{1}^{T} \cap
  S_{1}^{V}\biggr) &= \Oo_{\oM_{0,\vert{T}\vert+1} \times
                     \oM_{0,\vert{T^{c}}\vert+1}}\biggl(\oM_{0,\vert{T}\vert+1} \times
                     \biggl(-\sum_{\substack{\bullet \in W \\ 1,2 \in W^{c}}}
  D_{\vert{T^{c}}\vert+1}^{W}\biggr)\biggr) \\ 
                   &= L_{b}^{\vee}.
\end{align*}
Hence, we have shown that
\begin{equation*}
(\sigma_{k}^{i} \vert_{D^{T}})^{*}N_{S_{k}^{T}/B_{k}} = (\sigma_{k}^{i} \vert_{D^{T}})^{*}(g_{k} \vert_{S_{k}^{T}})^{*}((L_{a}^{\vee} \otimes L_{b}^{\vee}) \oplus L_{b}^{\vee}) = (L_{a}^{\vee} \otimes L_{b}^{\vee}) \oplus L_{b}^{\vee},
\end{equation*}
with the final equality following from Theorem \ref{Keel}.
\end{proof}

\subsection{The cobordism class of $\oM_{0,n}$}

For each $1 \le i \le n$, let $\psi_{i} = c_{1}(L_{i}) \in \Omega^{1}(\oM_{0,n})$, called the \textit{psi-classes}. On the divisor $D^{T}$ of $\oM_{0,n}$, we let $\psi_{a} = c_{1}(L_{a})$ and $\psi_{b} = c_{1}(L_{b})$, which are elements of $\Omega^{1}(D^{T})$, and let 
\begin{equation*}
i_{T} \colon D^{T} \to \oM_{0,n}
\end{equation*}
be the inclusion. Following the notation in \cite{SchubertCalculus}, write $q(x,y)$ for the unique power series satisfying $x+_{\Omega}y = x+y-xyq(x,y)$, and define the power series $\varphi(x) \in \L^{*}[[x]]$ by $\varphi(x) = q(x,\bar x)$. Note that we have the relation
\begin{equation} \label{VarphiDef}
\varphi(x) = \frac{1}{x}+\frac{1}{\bar x}.
\end{equation}
\begin{prop} \label{CobUniversalCurve}
As elements of $\Omega^{1}(\oM_{0,n})$, we have that
\begin{equation*}
[\oM_{0,n+1}\xrightarrow{\pi} \oM_{0,n}] = [\P^{1}][\oM_{0,n} \xrightarrow{\text{id}} \oM_{0,n}] +\sum_{T} (i_{T})_{*}\biggl(\frac{1}{\psi_{a}\bigl(\bar\psi_{a}+_{\Omega}\bar\psi_{b}\bigr)}
     - \frac{1}{\bigl( \bar\psi_{a}+_{\Omega}\bar\psi_{b} \bigr) \bar\psi_{b}} + \frac{1}{\bar\psi_a\bar\psi_b}\biggr).
\end{equation*}
\end{prop}
\begin{proof}
First, since the map $\pi_{n-2} \colon \oM_{0,n+1} \to B_{n-2}$ is an isomorphism over $B_{1}$ by Theorem \ref{Keel}, we have that
\begin{equation*}
[\oM_{0,n+1} \xrightarrow{\pi_{1}} B_{1}] = [B_{n-2} \xrightarrow{g_{n-2}} B_{1}].
\end{equation*}
For $\vert{T^{c}}\vert = k$, let $i_{k,T} \colon D^{T} \to B_{k}$ be the inclusion obtained by identifying $S_{k}^{T}$ with $D^{T}$ as in the statement of Lemma \ref{Conormals}. Since the $D^{T}$ for $\vert{T^{c}}\vert = k$ are disjoint, we have by Lemma \ref{Blowup} that 
\begin{align*}
[B_{k+1} \xrightarrow{f_{k+1}} B_{k}] &= [B_{k} \xrightarrow{\text{id}} B_{k}]-\sum_{\vert{T^{c}}\vert = k} (i_{k,T})_{*}[\P(N_{D^{T}/B_{k}} \oplus \Oo_{D^T}) \to D^{T}] \\
&\quad+ \sum_{\vert{T^{c}}\vert = k} (i_{k,T})_{*}p_{*} [\P_{\P(N_{D^{T}/B_{k}})}(\Oo(1) \oplus \Oo) \to \P(N_{D^{T}/B_{k}})].
\end{align*}
Pushing forward to $B_{1}$ via $g_{k}$, we get that
\begin{align*}
[B_{k+1} \xrightarrow{g_{k+1}} B_{1}] &= [B_{k} \xrightarrow{g_{k}} B_{1}]-\sum_{\vert{T^{c}}\vert = k} (g_{k})_{*}(i_{k,T})_{*}[\P(N_{D^{T}/B_{k}} \oplus \Oo_{D^T}) \to D^{T}] \\
&\quad+ \sum_{\vert{T^{c}}\vert = k} (g_{k})_{*}(i_{k,T})_{*}p_{*} [\P_{\P(N_{D^{T}/B_{k}})}(\Oo(1) \oplus \Oo) \to \P(N_{D^{T}/B_{k}})].
\end{align*}
By induction, and Lemma \ref{Conormals}, 
\begin{align*}
[B_{n-2} \xrightarrow{g_{n-2}} B_{1}] &= [B_{1} \xrightarrow{\text{id}} B_{1}]-\sum_{T} (g_{k})_{*}(i_{k,T})_{*}[\P((L_{a}^{\vee} \otimes L_{b}^{\vee}) \oplus L_{b}^{\vee} \oplus \Oo_{D^T}) \to D^{T}] \\
&\quad+ \sum_{T} (g_{k})_{*}(i_{k,T})_{*}p_{*} [\P_{\P((L_{a}^{\vee} \otimes L_{b}^{\vee}) \oplus L_{b}^{\vee})}(\Oo(1) \oplus \Oo) \to \P((L_{a}^{\vee} \otimes L_{b}^{\vee}) \oplus L_{b}^{\vee})].
\end{align*}
We then push forward to $\oM_{0,n}$ via the projection $p_{1} \colon B_{1} = \oM_{0,n} \times \P^{1} \to \oM_{0,n}$. Since $g_{k} \circ i_{k,T} = i_{T}$, it follows from the definition of the product in $\Omega^{*}$ that 
\begin{align}
\notag
\lbrack\oM_{0,n+1} \xrightarrow{\pi} \oM_{0,n}\rbrack &= [\P^{1}][\oM_{0,n} \xrightarrow{\text{id}} \oM_{0,n}]-\sum_{T} (i_{T})_{*}[\P((L_{a}^{\vee} \otimes L_{b}^{\vee}) \oplus L_{b}^{\vee} \oplus \Oo_{D^T}) \to D^{T}] \\
\label{equation0}
&\quad+ \sum_{T} (i_{T})_{*}p_{*} [\P_{\P((L_{a}^{\vee} \otimes L_{b}^{\vee}) \oplus L_{b}^{\vee})}(\Oo(1) \oplus \Oo) \to \P((L_{a}^{\vee} \otimes L_{b}^{\vee}) \oplus L_{b}^{\vee})].
\end{align}
By Theorem \ref{Vishik} applied to the rank-3 bundle $V = (L_{a}^{\vee} \otimes L_{b}^{\vee}) \oplus L_{b}^{\vee} \oplus \Oo_{D^T}$ on $X = D^T$, so in this case $\lambda_{1} = \bar\psi_{a}+_{\Omega}\bar\psi_{b}$, $\lambda_{2} = \bar\psi_{b}$ and $\lambda_{3} = 0$, we get that
\begin{equation} \label{equation1}
[\P((L_{a}^{\vee} \otimes L_{b}^{\vee}) \oplus L_{b}^{\vee} \oplus \Oo_{D^T}) \to D^{T}] = \frac{1}{\psi_{a}(\psi_{a}+_{\Omega}\psi_{b})}+\frac{1}{\bar\psi_{a}\psi_{b}}+\frac{1}{(\bar\psi_{a}+_{\Omega}\bar\psi_{b})\bar\psi_{b}}.
\end{equation}
Again by Theorem \ref{Vishik}, we have that
\begin{equation*}
[\P_{\P((L_{a}^{\vee} \otimes L_{b}^{\vee}) \oplus L_{b}^{\vee})}(\Oo(1) \oplus \Oo) \to \P((L_{a}^{\vee} \otimes L_{b}^{\vee}) \oplus L_{b}^{\vee})] = \frac{1}{\xi}+\frac{1}{\bar\xi},
\end{equation*}
where $\xi = c_{1}(\Oo(1))$, which is equal to $\varphi(\xi)$ by Equation \ref{VarphiDef}. Hence, by applying Theorem \ref{Vishik} to the rank-2 bundle $V = (L_{a}^{\vee} \otimes L_{b}^{\vee}) \oplus L_{b}^{\vee}$ on $X = D^{T}$, it follows that
\begin{multline} \label{equation2}
p_{*} [\P_{\P((L_{a}^{\vee} \otimes L_{b}^{\vee}) \oplus L_{b}^{\vee})}(\Oo(1) \oplus \Oo) \to
\P((L_{a}^{\vee} \otimes L_{b}^{\vee}) \oplus L_{b}^{\vee})] \\
= \frac{1}{\psi_{a}}\left(\frac{1}{\psi_{a}+_{\Omega}\psi_{b}}+\frac{1}{\bar\psi_{a}+_{\Omega}\bar\psi_{b}}\right)+\frac{1}{\bar\psi_{a}}\left(\frac{1}{\psi_{b}}+\frac{1}{\bar\psi_{b}}\right).
\end{multline}
Substituting Equations \ref{equation1} and \ref{equation2} into Equation \ref{equation0} yields our desired formula.
\end{proof}
As a power series in $\L^{*}[[x,y]]$, we have  
\begin{align}
\notag
\frac{1}{x(\bar x+_{\Omega}\bar y)}
&-\frac{1}{(\bar x +_{\Omega} \bar y)\bar y}+\frac{1}{\bar x \bar y} \\
\notag
&\quad= \frac{1}{x}(\varphi(x+_{\Omega}y)-\varphi(y))+\biggl(\frac{1}{y}-\varphi(y)\biggr)(\varphi(x+_{\Omega}y)-\varphi(x))+\frac{1}{x+_{\Omega}y}(\varphi(y)-q(x,y)) \\
\label{b's}
&\quad= \sum_{i,j \ge 0} b_{ij}x^{i}y^{j},
\end{align}
where $b_{ij}$ is in degree $2+i+j$. We prove the first equality, and
explain why it is a power series, in Proposition~\ref{Expression},
after which we give, in Equation~\ref{R(x,y)}, a computation of the
coefficients up to degree 3. In particular, using
Example~\ref{LowDegGen}, we know that
\begin{equation*}
  b_{00} = [\P^{1}]^{2}-[\P^{2}], \hspace{0.5cm} b_{10} =
  [\Bl_{\text{pt}}\P^{3}]-[\P^{1}][\P^{2}], \hspace{0.5cm}
  b_{01} = 2([\Bl_{\text{pt}}\P^{3}]-[\P^{1}][\P^{2}]).
\end{equation*}

If $X \in \Sm_{k}$ and $p \colon X \to \Spec k$ is the projection, then for $\alpha \in \Omega^{*}(X)$ we let 
\begin{equation*}
\int_{X}\alpha := p_{*}(\alpha).
\end{equation*}
\begin{cor} \label{M0n}
We have that $[\oM_{0,3}] = 1$, and for each $n \ge 3$, 
\begin{equation*}
[\oM_{0,n+1}] = [\P^{1}][\oM_{0,n}]+\sum_{i,j \ge 0}\sum_{T}b_{ij}\int_{\oM_{0,\vert{T}\vert+1}} \psi_{\bullet}^{i}\int_{\oM_{0,\vert{T^c}\vert+1}} \psi_{\bullet}^{j}.
\end{equation*}
\end{cor}
\begin{proof}
By pushing forward the relation of Proposition \ref{CobUniversalCurve} to a point, we have that
\begin{equation*}
[\oM_{0,n+1}] = [\P^{1}][\oM_{0,n}]+\sum_{i,j \ge 0}\sum_{T}b_{ij}\int_{D^T}\psi_{a}^{i}\psi_{b}^{j}.
\end{equation*}
Consider the diagram
\begin{equation*}
  \begin{tikzcd}[column sep=5em]
    D^{T} \ar[r,"\text{pr}_{2}"] \ar[d,"\text{pr}_{1}"'] \ar[dr,"p"]
    & \oM_{0,\vert{T^c}\vert+1} \ar[d,"q_{2}"] \\
    \oM_{0,\vert{T}\vert+1} \ar[r,"q_{1}"] & \Spec k.
  \end{tikzcd}
\end{equation*}
By the projection formula, we have that
\begin{equation*}
\int_{D^T}\psi_{a}^{i}\psi_{b}^{j} = p_{*}((\text{pr}_{1}^{*}\psi_{\bullet}^{i})(\text{pr}_{2}^{*}\psi_{\bullet}^{j})) = (q_{1})_{*}(\text{pr}_{1})_{*}((\text{pr}_{1}^{*}\psi_{\bullet}^{i})(\text{pr}_{2}^{*}\psi_{\bullet}^{j})) = (q_{1})_{*}(\psi_{\bullet}^{i}(\text{pr}_{1})_{*}\text{pr}_{2}^{*}\psi_{\bullet}^{j}).
\end{equation*}
Following from the axioms of algebraic cobordism \cite[Definition 1.1.2]{AlgCob}, we know that $(\text{pr}_{1})_{*}\text{pr}_{2}^{*} = q_{1}^{*}(q_{2})_{*}$. Hence, it follows that
\begin{equation*}
(q_{1})_{*}(\psi_{\bullet}^{i}(\text{pr}_{1})_{*}\text{pr}_{2}^{*}\psi_{\bullet}^{j}) = (q_{1})_{*}(\psi_{\bullet}^{i}q_{1}^{*}(q_{2})_{*}\psi_{\bullet}^{j}) = (q_{1})_{*}(\psi_{\bullet}^{i})(q_{2})_{*}(\psi_{\bullet}^{j}),
\end{equation*}
which is equal to $\left(\int_{\oM_{0,\vert{T}\vert+1}} \psi_{\bullet}^{i}\right)\left(\int_{\oM_{0,\vert{T^c}\vert+1}} \psi_{\bullet}^{j}\right)$.
\end{proof}

\section{The String equation in cobordism} \label{StringSection}

In this section, we give an inductive formula for all cobordism-valued psi-class intersections on $\oM_{0,n}$, which in particular includes the cobordism classes $[\oM_{0,n}]$. For integers $d_{1},\dots,d_{n} \ge 0$, define
\begin{equation*}
  \psi^d = \psi_1^{d_1} \dotsm \psi_n^{d_n} ,
\end{equation*}
and for $1\le i\le n$, define
\begin{equation*}
  \psi^d/\psi_i =
  \begin{cases}
    \psi_1^{d_1} \dotsm \psi_i^{d_i-1} \dotsm \psi_n^{d_n} , & d_i>0 , \\
    0 , & d_i=0 .
  \end{cases}
\end{equation*}
We let $|d|=d_1+\cdots+d_n$. Write
\begin{equation} \label{c's}
\frac{x}{\bar x} = x\varphi(x)-1 = \sum_{j \ge 0} c_{j}x^{j},
\end{equation}
where $c_{j}$ is in degree $j$. In low degrees, we have that $c_{0} = -1$, $c_{1} = [\P^{1}]$, and $c_{2} = 0$, which follows from Equation \ref{x/chi(x)} and Example \ref{LowDegGen}. 
\begin{thm}[String equation] \label{StringEquation}
For $n \ge 3$, we have  
  \begin{multline*}
    \int_{\oM_{0,n+1}} 1\cdot\psi^d \\
    \begin{aligned}
    &= [\P^{1}] \int_{\oM_{0,n}} \psi^d
      - \sum_{i=1}^{n}\int_{\oM_{0,n}} \frac{\psi_{i}}{\bar\psi_{i}} \psi^d/\psi_i
      + \sum_{T}\int_{D^{T}} \biggl(
      \frac{1}{\psi_{a}\bigl(\bar\psi_{a}+_{\Omega}\bar\psi_{b}\bigr)}
     - \frac{1}{\bigl( \bar\psi_{a}+_{\Omega}\bar\psi_{b} \bigr) \bar\psi_{b}} + \frac{1}{\bar\psi_a\bar\psi_b}
      \biggr) i_{T}^{*} \psi^d \\
    &= [\P^{1}] \int_{\oM_{0,n}}  \psi^d -\sum_{i=1}^{n} \sum_{j \ge 0} \int_{\oM_{0,n}}
       c_{j}\psi_{i}^{j} \psi^d/\psi_i + \sum_{i,j \ge 0} \sum_{T}
       b_{ij} \int_{\oM_{0,\vert{T}\vert+1}} \psi_{\bullet}^{i} \prod_{\ell \in T}\psi_{\ell}^{d_\ell}
      \int_{\oM_{0,\vert{T^c}\vert+1}} \psi_{\bullet}^{j} \prod_{\ell \in T^{c}} \psi_{\ell}^{d_\ell}.
    \end{aligned}
  \end{multline*}
\end{thm}
\begin{proof}
Let $I$ denote the set of all indices $1 \le i \le n$ such that $d_{i} > 0$. For each $i \in I$, we let $D_{i} = c_{1}(\Oo_{\oM_{0,n+1}}(D^{\{i,n+1\}}))$. By Equation \ref{ForgetfulRelation}, it follows that
\begin{equation*}
\psi_{i} = \pi^{*}\psi_{i}+_{\Omega}D_{i}.
\end{equation*}
Note that if $j \in I$ with $i \neq j$, then $D^{\{i,n+1\}} \cap D^{\{j,n+1\}} = \emptyset$, which implies that $D_{i}D_{j} = 0$. Indeed, let $f \colon D^{\{i,n+1\}} \to \oM_{0,n+1}$ denote the inclusion, so that $D_{i} = f_{*}(1)$. Since the canonical section $s$ of $\Oo(D^{\{j,n+1\}})$ has zero locus $D^{\{j,n+1\}}$, the pull-back section $f^{*}s$ has zero locus $D^{\{i,n+1\}} \cap D^{\{j,n+1\}} = \emptyset$. That is, the line bundle $f^{*}\Oo(D^{\{j,n+1\}})$ on $D^{\{i,n+1\}}$ is trivial as it has a nowhere vanishing section, and so $c_{1}(f^{*}\Oo(D^{\{j,n+1\}})) = 0$. Hence, by the projection formula,
\begin{equation*}
D_{i}D_{j} = f_{*}(1)c_{1}(\Oo(D^{\{j,n+1\}})) = f_{*}(f^{*}c_{1}(\Oo(D^{\{j,n+1\}}))) = 0.
\end{equation*}
Similarly, note that $f^{*}L_{i}$ is a trivial line bundle on $D^{\{i,n+1\}}$, so
\begin{equation*}
\psi_{i}D_{i} = c_{1}(L_{i})f_{*}(1) = f_{*}(c_{1}(f^{*}L_{i})) = 0.
\end{equation*}
Hence, it follows that
\begin{align*}
\int_{\oM_{0,n+1}} 1\cdot\psi^d &= \int_{\oM_{0,n+1}} \prod_{i \in I} \psi_{i}(\pi^{*}\psi_{i}+_{\Omega}D_{i})^{d_{i}-1} \\
&= \int_{\oM_{0,n+1}} \prod_{i \in I} \psi_{i}(\pi^{*}\psi_{i})^{d_{i}-1}.
\end{align*}
Writing $\psi_{i} = \pi^{*}\psi_{i}+_{\Omega}D_{i}$ and $x+_{\Omega}y = x+y+\sum_{j,\ell \ge 1} a_{j\ell}x^{j}y^{\ell}$, we again have by the projection formula that
\begin{align*}
\int_{\oM_{0,n+1}} \prod_{i \in I} \psi_{i}(\pi^{*}\psi_{i})^{d_{i}-1} &= \int_{\oM_{0,n+1}} \prod_{i \in I} \biggl(\pi^{*}\psi_{i}+D_{i}+\sum_{j,\ell \ge 1} a_{j\ell}(\pi^{*}\psi_{i})^{j}D_{i}^{\ell}\biggr)(\pi^{*}\psi_{i})^{d_{i}-1} \\
&= \int_{\oM_{0,n}} \pi_{*}(1)\psi^d+\sum_{i=1}^{n} \int_{\oM_{0,n}}
  \biggl(\pi_{*}(D_{i})+\sum_{j,\ell \ge 1}
  a_{j\ell}\psi_{i}^{j}\pi_{*}(D_{i}^{\ell})\biggr)\psi^d/\psi_i.
\end{align*}
For $\ell \ge 2$, we have that
\begin{equation*}
D_{i}^{\ell} = f_{*}(1)f_{*}(1)^{\ell-1} = f_{*}(f^{*}(f_{*}(1)^{\ell-1})) = f_{*}(f^{*}(c_{1}(\Oo(D^{\{i,n+1\}}))^{\ell-1})) = f_{*}(f^{*}c_{1}(\Oo(D^{\{i,n+1\}})))^{\ell-1}.
\end{equation*}
Note that we can identify $f$ with the section $\sigma^{i}$ of $\pi \colon \oM_{0,n+1} \to \oM_{0,n}$ corresponding to the $i^{\text{th}}$ marked point, and so $f^{*}c_{1}(\Oo(D^{\{i,n+1\}})) = L_{i}^{\vee}$. Hence, it follows that $D_{i}^{\ell} = f_{*}(\bar\psi_{i}^{\ell-1})$, and 
\begin{equation*}
\pi_{*}(D_{i}^{\ell}) = \pi_{*}f_{*}(\bar\psi_{i}^{\ell-1}) = \bar\psi_{i}^{\ell-1}.
\end{equation*}
Observe that this equation also holds for $k = 1$. Therefore, we have
\begin{align*}
  \int_{\oM_{0,n+1}} 1 \cdot \psi^d &= \int_{\oM_{0,n}} \pi_{*}(1)\psi^d+\sum_{i=1}^{n}
                          \int_{\oM_{0,n}} \biggl(1+\sum_{j,\ell \ge 1}
                          a_{j\ell}\psi_{i}^{j}\bar\psi_{i}^{\ell-1}\biggr)\psi^d/\psi_i \\
                          &= \int_{\oM_{0,n}} \pi_{*}(1)\psi^d-\sum_{i=1}^{n}
                            \int_{\oM_{0,n}} \frac{\psi_{i}}{\bar\psi_{i}}\psi^d/\psi_i,
\end{align*}
with the last equality following from the equation $\psi_{i}+_{\Omega}\bar\psi_{i} = 0$. We then replace $\pi_{*}(1)$ using Proposition \ref{CobUniversalCurve} and use the projection formula. The second equality in the statement of Theorem \ref{StringEquation} follows from the fact that
\begin{equation*}
i_{T}^{*}\psi_{j} = \begin{cases}
\text{pr}_{1}^{*}\psi_{j} &\text{if $j \in T$,}\\
\text{pr}_{2}^{*}\psi_{j} &\text{if $j \in T^{c}$,} \\
\end{cases} 
\end{equation*} 
and the same argument as in Corollary \ref{M0n}.
\end{proof}
\begin{cor}
Together with the initial condition $\int_{\oM_{0,3}} 1 = 1$, Theorem \ref{StringEquation} determines the value of all psi-class intersections $\int_{\oM_{0,n}} \psi^d$. 
\end{cor}
\begin{proof}
Since $\text{dim}\,\oM_{0,n} = n-3$, we have by the dimension-vanishing property of $\Omega^{*}$ that $\psi^d = 0$ if $|d| > n-3$. If $0 \le |d| \le n-3$, then $d_{i} = 0$ for some $1 \le i \le n$, and hence Theorem \ref{StringEquation} can be applied. 
\end{proof}
\begin{exmp} \label{ExmpCalcs}
Using Theorem \ref{StringEquation}, and in particular Corollary \ref{M0n}, together with the presentation of the universal formal group law given by Equation \ref{UnivFGL} in Appendix \ref{Appendix}, we obtain the following expressions for the classes $[\oM_{0,n}]$ for $n \le 8$:
\begin{itemize}
\item[] $[\oM_{0,3}] = 1$
\item[] $[\oM_{0,4}] = u_{1}$
\item[] $[\oM_{0,5}] = 4u_{1}^{2}-3u_{2}$
\item[] $[\oM_{0,6}] = 31u_{1}^{3}-30u_{1}u_{2}+17u_{3}$
\item[] $[\oM_{0,7}] = 273u_{1}^{4}-317u_{1}^{2}u_{2}+70u_{2}^{2}+214u_{1}u_{3}-25u_{4}$
\item[] $[\oM_{0,8}] = 2898u_{1}^{5}-4063u_{1}^{3}u_{2}+2012u_{1}u_{2}^{2}+2765u_{1}^{2}u_{3}-1204u_{2}u_{3}-385u_{1}u_{4}+461u_{5}$.
\end{itemize}
By calculating Chern numbers using the expressions for the generators $u_i$ given in Example \ref{LowDegGen}, we can alternatively express these cobordism classes in the rational basis $p_{m} := [\P^{m}]$ of the cobordism ring given by projective spaces. We have, which can also be derived from the formalism of Coates-Givental \cite{CoatesGivental}:
\begin{itemize}
\item[] $[\oM_{0,3}] = 1$
\item[] $[\oM_{0,4}] = p_{1}$
\item[] $[\oM_{0,5}] = 4p_{1}^{2}-3p_{2}$
\item[] $[\oM_{0,6}] = \frac{45}{2}p_{1}^{3}-30p_{1}p_{2}+\frac{17}{2}p_{3}$
\item[] $[\oM_{0,7}] = 166p_{1}^{4}-317p_{1}^{2}p_{2}+70p_{2}^{2}+107p_{1}p_{3}-25p_{4}$
\item[] $[\oM_{0,8}] = \frac{3031}{2}p_{1}^{5}-\frac{11305}{3}p_{1}^{3}p_{2}+\frac{3563}{2}p_{1}p_{2}^{2}+\frac{2765}{2}p_{1}^{2}p_{3}-602p_{2}p_{3}-385p_{1}p_{4}+\frac{461}{6}p_{5}$.
\end{itemize}
\end{exmp}

\section{The Chow ring and $K$-theory} \label{KTheoryChow}

In this section we analyze the image, using the universality of algebraic cobordism, of Theorem \ref{StringEquation} in the well-known cases of the Chow ring and $K$-theory, due to Witten \cite{Witten} in the Chow/cohomology case, and Lee's work \cite{Lee,Lee2} in the case of $K$-theory. 
\begin{rem}
Since algebraic cobordism is the universal oriented cohomology theory \cite{AlgCob}, one can obtain a string equation for psi-class intersections on $\oM_{0,n}$ taking values in any oriented cohomology theory $\text{A}^{*}$, by applying the orientation map $\L^{*} \to \text{A}^{*}(\Spec k)$ to Theorem \ref{StringEquation}.
\end{rem}

\subsection{The Chow ring}
The Chow ring $\text{CH}^{*}$ has associated formal group law $x+_{\text{CH}}y = x+y$, known as the additive formal group law, and in this case $\bar x = -x$. The universality of $\Omega^{*}$ yields a map $\L^{*} \to \text{CH}^{*}(\Spec k)$ given by $[X] \mapsto 0$ if $\text{dim}\,X > 0$. The string equation of Theorem \ref{StringEquation} in the case of the additive formal group law is
\begin{equation*}
\int_{\oM_{0,n+1}} 1 \cdot \psi^d = \sum_{i=1}^{n} \int_{\oM_{0,n}}\psi^d/\psi_i.
\end{equation*}
This recursion has the closed-form solution
\begin{equation*}
  \int_{\oM_{0,n}} \psi^d = \binom{n-3}{d_{1},\dots,d_{n}} .
\end{equation*}
We can re-write this result in terms of generating functions as
follows. Let
\begin{equation*}
  E_n = \sum_d t^d \int_{\oM_{0,n}} \psi^d
\end{equation*}
for each $n \ge 3$, where $t^d = t_1^{d_1}\dotsm t_n^{d_n}$. Then, we
have that $E_{n} = (t_{1}+\dots+t_{n})^{n-3}$.

\subsection{$K$-theory}

Let $K^{0}(X)$ be the Grothendieck group of locally free coherent sheaves on $X$, and $K^{0}(X)[\beta,\beta^{-1}] := K^{0}(X) \otimes_{\Z} \Z[\beta,\beta^{-1}]$, where $\beta$ is a formal variable of degree $-1$. We denote by $[\E]$ the class in the Grothendieck group of a locally free coherent sheaf $\E$. The formal group law associated to $K^{0}[\beta,\beta^{-1}]$ is the multiplicative formal group law $x+_{K}y = x+y-\beta xy$, and $\bar x = -x/(1-\beta x)$. In this case, if $f \colon Y \to X$ is a projective morphism of relative codimension $d$, then
\begin{equation*}
f_{\ast}([\E] \beta^{n}) = \sum_{i\ge0} (-1)^{i}[R^{i}f_{\ast}(\E)] \beta^{n-d}.
\end{equation*}
Hence, if $L \to X$ is a line bundle and $\Ll$ its sheaf of sections, then $c_{1}(L) = (1-[\Ll^{\vee}])\beta^{-1}$. The universality of $\Omega^{*}$ yields a map $\L^{*} \to K^{0}(\Spec k)[\beta,\beta^{-1}]$ given by 
\begin{equation*}
[X] \mapsto \chi_{\text{alg}}(X,\Oo_{X})\beta^{\text{dim}\,X},
\end{equation*}
where $\chi_{\text{alg}}(X,\Oo_{X}) = \sum_{i \ge 0} (-1)^{i}\text{dim}\,H^{i}(X,\Oo_{X})$ is the algebraic Euler characteristic. The string equation of Theorem \ref{StringEquation} in the case of the multiplicative formal group law is
\begin{equation*}
\int_{\oM_{0,n+1}} 1 \cdot \psi^d = \beta\int_{\oM_{0,n}} \psi^d+\sum_{i=1}^{n} \int_{\oM_{0,n}} (1-\beta\psi_{i})\psi^d/\psi_i.
\end{equation*}
However, if we consider certain twists of the psi-classes, we obtain an alternative recursion with a simple closed form solution. Namely, for each $1 \le i \le n$, let 
\begin{equation*}
\Psi_{i} = [\Ll_{i}]\psi_{i} = ([\Ll_{i}]-1)\beta^{-1},
\end{equation*}
where $\Ll_{i}$ is the sheaf of sections of the line bundle $L_{i} \to \oM_{0,n}$. 
\begin{prop} \label{KtheoryString}
For $n \ge 3$, we have 
\begin{equation*}
\int_{\oM_{0,n+1}} 1\cdot \Psi^d = \beta\int_{\oM_{0,n}} \Psi^d+\sum_{i=1}^{n} \int_{\oM_{0,n}} \Psi^d/\Psi_i.
\end{equation*}
\end{prop}
\begin{proof}
The structure of the proof is the same as that of Theorem \ref{StringEquation}. However, if we again let $D_{i} = c_{1}(\Oo_{\oM_{0,n+1}}(D^{\{i,n+1\}})$), and then define $\tilde{D}_{i} = [\Oo(D^{\{i,n+1\}})]D_{i}$, we now have that
\begin{equation*}
\Psi_{i} = \pi^{*}\Psi_{i}+_{\tilde{K}}\tilde{D}_{i},
\end{equation*}
where $x+_{\tilde{K}}y = x+y+\beta xy$. Note that $\tilde{D}_{i}\tilde{D}_{j} = 0$ and $\Psi_{i}\tilde{D}_{i} = 0$ for $i \neq j$ since $D_{i}D_{j} = 0$ and $\psi_{i}D_{i} = 0$ as in the proof of Theorem \ref{StringEquation}, so we similarly get 
\begin{equation*}
\int_{\oM_{0,n+1}} 1\cdot \Psi^d = \int_{\oM_{0,n}} \pi_{*}(1)\Psi^d+\sum_{i=1}^{n} \int_{\oM_{0,n}} \left(\pi_{*}(\tilde{D}_{i})+\beta\Psi_{i}\pi_{*}(\tilde{D}_{i})\right)\Psi^d/\Psi_i.
\end{equation*}
We know that $\pi_{*}(1) = \beta$ by Proposition \ref{CobUniversalCurve}. Letting $f \colon D^{\{i,n+1\}} \to \oM_{0,n+1}$ be the inclusion, we have
\begin{equation*}
\pi_{*}(\tilde{D}_{i}) = \pi_{*}([\Oo(D^{\{i,n+1\}})]f_{*}(1)) = (\pi_{*} \circ f_{*})([f^{*}\Oo(D^{\{i,n+1\}})]) = [\Ll_{i}^{\vee}].
\end{equation*}
However, note that
\begin{equation*}
\Psi_{i}+_{\tilde{K}}(([\Ll_{i}^\vee]-1)\beta^{-1}) = 0,
\end{equation*}
which implies 
\begin{equation*}
([\Ll_{i}^{\vee}]-1)\beta^{-1} = -\frac{\Psi_{i}}{1+\beta\Psi_{i}} = \sum_{j \ge 1} (-1)^{j}\beta^{j-1}\Psi_{i}^{j},
\end{equation*}
and hence
\begin{equation*}
\pi_{*}(\tilde{D}_{i}) = 1+\sum_{j \ge 1}(-1)^{j}\beta^{j}\Psi_{i}^{j}.
\end{equation*}
Therefore, it follows that
\begin{align*}
  \int_{\oM_{0,n+1}} 1\cdot \Psi^d
&= \beta\int_{\oM_{0,n}} \Psi^d+\sum_{i=1}^{n}\biggl(1+\sum_{j \ge 1}(-1)^{j}\beta^{j}\Psi_{i}^{j}+\beta\Psi_{i}+\sum_{j \ge 1}(-1)^{j}\beta^{j+1}\Psi_{i}^{j+1}\biggr)\Psi^d/\Psi_i \\
&= \beta\int_{\oM_{0,n}} \Psi^d+\sum_{i=1}^{n} \int_{\oM_{0,n}} \Psi^d/\Psi_i. 
  \tag*{\qedhere}
\end{align*}
\end{proof}

For each $n \ge 3$, define the generating function
\begin{equation*}
E_{n} = \sum_d t^d \int_{\oM_{0,n}} \Psi^d .
\end{equation*}
\begin{cor}
We have that
\begin{equation*}
\int_{\oM_{0,n}} \Psi^d = \beta^{n-3-|d|} \binom{n-3}{n-3-|d|,d_{1},\dots,d_{n}},
\end{equation*}
and hence $E_{n} = (\beta+t_{1}+\dots+t_{n})^{n-3}$. 
\end{cor}
\begin{proof}
Since $\psi^d = 0$ for $|d| > n-3$, it follows that $\Psi^d = 0$ for $|d| > n-3$. Together with the initial condition $\int_{\oM_{0,3}} 1 = 1$, the result follows from Proposition \ref{KtheoryString} by induction. 
\end{proof}
\begin{rem}
Observe that more generally in $\Omega^{*}$, if $|d| = n-3$, then
\begin{equation*}
  \int_{\oM_{0,n}} \psi^d = \binom{n-3}{d_{1},\dots,d_{n}},
\end{equation*}
which follows since all terms in the cobordism-valued string equation that do not appear in the Chow ring string equation vanish in this case. This phenomenon can be seen in the tables in Section \ref{Tables}.
\end{rem}

\appendix

\section{Formal group law calculations} \label{Appendix}

In this section we perform some low-degree calculations with the universal formal group law $x+_{\Omega}y \in \L^{*}[[x,y]]$. Recall that we define $q(x,y)$ to be the unique power series satisfying $x+_{\Omega}y = x+y-xyq(x,y)$, and we define $\varphi(x) = q(x,\bar x)$. We now give a combinatorial proof that the expression appearing in the statement of Proposition \ref{CobUniversalCurve} (and Theorem \ref{StringEquation}) is indeed a power series. 
\begin{prop} \label{Expression}
We have that
\begin{align*}
&\frac{1}{x(\bar x+_{\Omega}\bar y)}-\frac{1}{(\bar x +_{\Omega} \bar y)\bar y}+\frac{1}{\bar x \bar y} \\ 
&\quad= \frac{1}{x}(\varphi(x+_{\Omega}y)-\varphi(y))+\biggl(\frac{1}{y}-\varphi(y)\biggr)(\varphi(x+_{\Omega}y)-\varphi(x))+\frac{1}{x+_{\Omega}y}(\varphi(y)-q(x,y)) \\
\end{align*}
is a power series with coefficients in the Lazard ring $\L^{*}$. 
\end{prop}
\begin{proof}
First, using Equation \ref{VarphiDef}, we have that
\begin{align*}
&\frac{1}{x(\bar x+_{\Omega}\bar y)}-\frac{1}{(\bar x +_{\Omega} \bar y)\bar y}+\frac{1}{\bar x \bar y} \\
&\quad= \left(\varphi(x+_{\Omega}y)-\frac{1}{x+_{\Omega}y}\right)\left(\frac{1}{x}-\varphi(y)+\frac{1}{y}\right)+\left(\varphi(x)-\frac{1}{x}\right)\left(\varphi(y)-\frac{1}{y}\right) \\
&\quad= \frac{1}{x}(\varphi(x+_{\Omega}y)-\varphi(y))+\biggl(\frac{1}{y}-\varphi(y)\biggr)(\varphi(x+_{\Omega}y)-\varphi(x))+\frac{\varphi(y)}{x+_{\Omega}y}+\frac{1}{xy}-\frac{1}{x(x+_{\Omega}y)}-\frac{1}{y(x+_{\Omega}y)}.
\end{align*}
The sum of the last three terms is equal to
\begin{equation*}
\frac{(x+_{\Omega}y)-x-y}{xy(x+_{\Omega}y)} = -\frac{q(x,y)}{x+_{\Omega}y},
\end{equation*}
and hence
\begin{align*}
&\frac{1}{x(\bar x+_{\Omega}\bar y)}-\frac{1}{(\bar x +_{\Omega} \bar y)\bar y}+\frac{1}{\bar x \bar y} \\ 
&\quad= \frac{1}{x}(\varphi(x+_{\Omega}y)-\varphi(y))+\biggl(\frac{1}{y}-\varphi(y)\biggr)(\varphi(x+_{\Omega}y)-\varphi(x))+\frac{1}{x+_{\Omega}y}(\varphi(y)-q(x,y)). \\
\end{align*}
Observe that $\varphi(x+_{\Omega}y)-\varphi(y)$ is divisible by $x$. Indeed, we know that it is divisible by $(x+_{\Omega}y)-y$, but $(x+_{\Omega}y)-y = x(1-yq(x,y))$. Similarly, we have that $\varphi(x+_{\Omega}y)-\varphi(x)$ is divisible by $y$. To see that $\varphi(x)-q(x,y)$ is divisible by $x+_{\Omega}y$, note that 
\begin{equation*}
\varphi(x)-q(x,y) = q(x,\bar x)-q(x,y)
\end{equation*}
is divisible by $\bar x-y$, but since $y = (x+_{\Omega}y)+_{\Omega} \bar x$, we have that 
\begin{equation*}
  \bar x-y = (x+_{\Omega}y)(-1+\bar xq(x+_{\Omega}y,\bar x)).
  \qedhere
\end{equation*}
\end{proof}
We now give the following presentation \cite{Hazewinkel} of the
coefficients of the universal formal group law $x+_{\Omega}y$ up to
homogenous degree 6. See Stapleton's website \cite{Stapleton} for
significant further computations of the coefficients of $x+_{\Omega}y$.
\begin{align}
\notag
x+_{\Omega}y &= x+y-u_{1}xy+(u_{1}^{2}-u_{2})(x^{2}y+xy^{2}) \\
\notag
&\quad+(-2u_{1}^{3}+2u_{1}u_{2}-2u_{3})(x^{3}y+xy^{3})+(-4u_{1}^{3}+4u_{1}u_{2}-3u_{3})x^{2}y^{2} \\
\notag
&\quad+(3u_{1}^{4}-3u_{1}^{2}u_{2}+u_{2}^{2}+4u_{1}u_{3}-u_{4})(x^{4}y+xy^{4}) \\
\notag
&\quad+(10u_{1}^{4}-11u_{1}^{2}u_{2}+3u_{2}^{2}+11u_{1}u_{3}-2u_{4})(x^{3}y^{2}+x^{2}y^{3}) \\
\notag
&\quad+(-4u_{1}^{5}+2u_{1}^{3}u_{2}-6u_{1}u_{2}^{2}-6u_{1}^{2}u_{3}+4u_{2}u_{3}+2u_{1}u_{4}-6u_{5})(x^{5}y+xy^{5})\\
\notag
&\quad+(-21u_{1}^{5}+21u_{1}^{3}u_{2}-22u_{1}u_{2}^{2}-28u_{1}^{2}u_{3}+15u_{2}u_{3}+7u_{1}u_{4}-15u_{5})(x^{4}y^{2}+x^{2}y^{4})\\
\label{UnivFGL}
&\quad+(-34u_{1}^{5}+37u_{1}^{3}u_{2}-33u_{1}u_{2}^{2}-43u_{1}^{2}u_{3}+22u_{2}u_{3}+10u_{1}u_{4}-20u_{5})x^{3}y^{3}+\dots
\end{align}
It follows that
\begin{align*}
q(x,y) &= u_{1}-(u_{1}^{2}-u_{2})(x+y) -(-2u_{1}^{3}+2u_{1}u_{2}-2u_{3})(x^{2}+y^{2})-(-4u_{1}^{3}+4u_{1}u_{2}-3u_{3})xy \\
&\quad-(3u_{1}^{4}-3u_{1}^{2}u_{2}+u_{2}^{2}+4u_{1}u_{3}-u_{4})(x^{3}+y^{3}) \\
&\quad-(10u_{1}^{4}-11u_{1}^{2}u_{2}+3u_{2}^{2}+11u_{1}u_{3}-2u_{4})(x^{2}y+xy^{2}) \\
&\quad-(-4u_{1}^{5}+2u_{1}^{3}u_{2}-6u_{1}u_{2}^{2}-6u_{1}^{2}u_{3}+4u_{2}u_{3}+2u_{1}u_{4}-6u_{5})(x^{4}+y^{4})\\
&\quad-(-21u_{1}^{5}+21u_{1}^{3}u_{2}-22u_{1}u_{2}^{2}-28u_{1}^{2}u_{3}+15u_{2}u_{3}+7u_{1}u_{4}-15u_{5})(x^{3}y+xy^{3})\\
&\quad-(-34u_{1}^{5}+37u_{1}^{3}u_{2}-33u_{1}u_{2}^{2}-43u_{1}^{2}u_{3}+22u_{2}u_{3}+10u_{1}u_{4}-20u_{5})x^{2}y^{2}+\dots
\end{align*}
Next, we have that
\begin{equation*}
\bar x = -x-u_{1}x^{2}-u_{1}^{2}x^{3}+(-2u_{1}^{3}+u_{1}u_{2}-u_{3})x^{4}+\dots
\end{equation*}
which implies that
\begin{align}
\notag
\varphi(x) = q(x,\bar x) &= u_{1}+(u_{1}^{3}-u_{1}u_{2}+u_{3})x^{2}+(u_{1}^{4}-u_{1}^{2}u_{2}+u_{1}u_{3})x^{3} \\
\label{varphi}
&\quad+(3u_{1}^{5}-2u_{1}^{3}u_{2}+2u_{1}u_{2}^{2}+4u_{1}^{2}u_{3}-u_{2}u_{3}-u_{1}u_{4}+2u_{5})x^{4}+\dots
\end{align}
and
\begin{align}
\notag
\frac{x}{\bar x} &= -1+x\varphi(x) \\
\notag
&=-1+u_{1}x+(u_{1}^{3}-u_{1}u_{2}+u_{3})x^{3}+(u_{1}^{4}-u_{1}^{2}u_{2}+u_{1}u_{3})x^{4}\\
\label{x/chi(x)}
&\quad+(3u_{1}^{5}-2u_{1}^{3}u_{2}+2u_{1}u_{2}^{2}+4u_{1}^{2}u_{3}-u_{2}u_{3}-u_{1}u_{4}+2u_{5})x^{5}+\dots
\end{align}
We now calculate the power series of Proposition \ref{Expression} up to order 3. First, we have that
\begin{align*}
\varphi(x+_{\Omega}y) &= u_{1}+(u_{1}^{3}-u_{1}u_{2}+u_{3})(x^{2}+2xy+y^{2})+(u_{1}^{4}-u_{1}^{2}u_{2}+u_{1}u_{3})(x^{3}+x^{2}y+xy^{2}+y^{3}) \\
&\quad+(3u_{1}^{5}-2u_{1}^{3}u_{2}+2u_{1}u_{2}^{2}+4u_{1}^{2}u_{3}-u_{2}u_{3}-u_{1}u_{4}+2u_{5})(x^{4}+y^{4}) \\
&\quad+(11u_{1}^{5}-9u_{1}^{3}u_{2}+10u_{1}u_{2}^{2}+15u_{1}^{2}u_{3}-6u_{2}u_{3}-4u_{1}u_{4}+8u_{5})(x^{3}y+xy^{3}) \\
&\quad+(17u_{1}^{5}-15u_{1}^{3}u_{2}+16u_{1}u_{2}^{2}+23u_{1}^{2}u_{3}-10u_{2}u_{3}-6u_{1}u_{4}+12u_{5})x^{2}y^{2}+\dots
\end{align*}
which implies, by Proposition \ref{Expression}, that
\begin{align}
\notag
&\frac{1}{x(\bar x+_{\Omega}\bar y)}-\frac{1}{(\bar x +_{\Omega} \bar y)\bar y}+\frac{1}{\bar x \bar y} \\
\notag
&\quad= \frac{1}{x}(\varphi(x+_{\Omega}y)-\varphi(y))+\biggl(\frac{1}{y}-\varphi(y)\biggr)(\varphi(x+_{\Omega}y)-\varphi(x))+\frac{1}{x+_{\Omega}y}(\varphi(y)-q(x,y)) \\
\notag
&\quad=(u_{1}^{2}-u_{2})+(u_{1}^{3}-u_{1}u_{2}+u_{3})(x+2y) \\
\notag
&\quad\quad+(5u_{1}^{4}-5u_{1}^{2}u_{2}+u_{2}^{2}+6u_{1}u_{3}-u_{4})(x^{2}+y^{2}) +(4u_{1}^{4}-4u_{1}^{2}u_{2}+u_{2}^{2}+5u_{1}u_{3}-u_{4})xy \\
\notag
&\quad\quad+(10u_{1}^{5}-9u_{1}^{3}u_{2}+6u_{1}u_{2}^{2}+13u_{1}^{2}u_{3}-3u_{2}u_{3}-3u_{1}u_{4}+4u_{5})x^{3} \\
\notag
&\quad\quad+(17u_{1}^{5}-15u_{1}^{3}u_{2}+15u_{1}u_{2}^{2}+23u_{1}^{2}u_{3}-9u_{2}u_{3}-6u_{1}u_{4}+11u_{5})x^{2}y \\
\notag
&\quad\quad+(14u_{1}^{5}-12u_{1}^{3}u_{2}+12u_{1}u_{2}^{2}+19u_{1}^{2}u_{3}-7u_{2}u_{3}-5u_{1}u_{4}+9u_{5})xy^{2} \\
\label{R(x,y)}
&\quad\quad+(12u_{1}^{5}-10u_{1}^{3}u_{2}+8u_{1}u_{2}^{2}+16u_{1}^{2}u_{3}-4u_{2}u_{3}-4u_{1}u_{4}+6u_{5})y^{3} +\dots
\end{align}

\begin{rem} \label{AppendixRemark}
When calculating the cobordism classes $[\P_{\P^{2}}(\Oo(1) \oplus \Oo \oplus \Oo)]$ and $[\P_{\P^{3}}(\Oo(1) \oplus \Oo \oplus \Oo)]$ to find expressions for the generators $u_{4}$ and $u_{5}$ as in Example \ref{LowDegGen}, we use that after setting $y = z = 0$ in the power series
\begin{equation*}
\frac{1}{(y+_{\Omega}\bar x)(z+_{\Omega}\bar x)}+\frac{1}{(x+_{\Omega}\bar y)(z+_{\Omega}\bar y)}+\frac{1}{(x+_{\Omega}\bar z)(y+_{\Omega}\bar z)},
\end{equation*}
we get the power series in one variable
\begin{align*}
&\varphi(x)^{2}-\frac{1}{x}(2\varphi(x)-q(x,0)-u_{1}) \\
&\quad= u_{2}+(-3u_{1}^{4}+3u_{1}^{2}u_{2}-u_{2}^{2}-4u_{1}u_{3}+u_{4})x^{2}+(2u_{1}u_{2}^{2}-2u_{2}u_{3}+2u_{5})x^{3}+\dots
\end{align*}
\end{rem}

\subsection{Tables} \label{Tables} Using Theorem \ref{StringEquation},
we calculate all cobordism-valued psi-class intersections
$\int_{\oM_{0,n}} \psi^d$ for $n \le 8$, given in the following tables, which
we organize by the total degree $|d|$. Note that we do not include
those intersections which can be obtained from those in the tables by
permuting the markings.

\renewcommand{\thetable}{A.0}
\begin{table}[H] 
\caption[$|d| = 0$]{$|d| = 0$}
\label{Table0}
\begin{center}
{\renewcommand{\arraystretch}{1.2}
\begin{tabular}{|c|c|}
         \hline
         $n$
         & 1
         \\ \hline
         3
         &
         1
         \\ 
         4 
         &
         $u_1$
         \\ 
         5 
         &
         $4u_1^2 - 3u_2$
         \\ 
         6 
         &
         $31u_1^3 - 30u_1u_2+17u_3$
         \\ 
         7 
         &
         $273u_1^4 - 317u_1^2u_2+70u_2^2+214u_1u_3 - 25u_4$
         \\ 
         8 
         &
         $2898u_{1}^{5}-4063u_{1}^{3}u_{2}+2012u_{1}u_{2}^{2}+2765u_{1}^{2}u_{3}-1204u_{2}u_{3}-385u_{1}u_{4}+461u_{5}$
         \\ \hline
\end{tabular} }
\end{center}
\end{table}

\renewcommand{\thetable}{A.1}
\begin{table}[H] 
\caption[$|d| = 1$]{$|d| = 1$}
\label{Table1}
\begin{center}
{\renewcommand{\arraystretch}{1.2}
\begin{tabular}{|c|c|}
         \hline
         $n$
         & $\psi_1$
         \\ \hline
         4 
         &
         1
         \\ 
         5 
         &
         $u_1$
         \\ 
         6 
         &
         $10u_1^2 - 9u_2$
         \\ 
         7 
         &
         $101u_1^3 - 100u_1u_2+67u_3$
         \\ 
         8 
         &
         $1078u_{1}^{4}-1302u_{1}^{2}u_{2}+350u_{2}^{2}+889u_{1}u_{3}-125u_{4}$
         \\ \hline
\end{tabular} }
\end{center}
\end{table}

\renewcommand{\thetable}{A.2}
\begin{table}[H] 
\caption[$|d| = 2$]{$|d| = 2$}
\label{Table2}
\begin{center}
{\renewcommand{\arraystretch}{1.2}
\begin{tabular}{|c|c|c|}
         \hline
         $n$
         & $\psi_1^2$
         & $\psi_1\psi_2$
         \\ \hline
         5 
         &
         1
         &
         2
         \\ 
         6 
         &
         $u_1$
         &
         0
         \\ 
         7 
         &
         $20u_1^2 - 19u_2$
         &
         $38u_1^2 - 36u_2$
         \\ 
         8 
         &
         $246u_{1}^{3}-245u_{1}u_{2}+177u_{3}$
         &
         $400u_{1}^{3}-400u_{1}u_{2}+330u_{3}$
         \\ \hline
\end{tabular} }
\end{center}
\end{table}

\renewcommand{\thetable}{A.3}
\begin{table}[H] 
\caption[$|d| = 3$]{$|d| = 3$}
\label{Table3}
\begin{center}
{\renewcommand{\arraystretch}{1.2}
\begin{tabular}{|c|c|c|c|}
         \hline
         $n$
         & $\psi_1^3$
         & $\psi_1^2\psi_2$
         & $\psi_1\psi_2\psi_3$
         \\ \hline
         6 
         &
         1
         &
         3
         &
         6
         \\ 
         7 
         &
         $u_1$
         &
         $-2u_1$
         &
         $-12u_1$
         \\ 
         8 
         &
         $35u_1^2-34u_2$
         &
         $100u_1^2-95u_2$
         &
         $210u_1^2-180u_2$
         \\ \hline
\end{tabular} }
\end{center}
\end{table}

\renewcommand{\thetable}{A.4}
\begin{table}[H] 
\caption[$|d| = 4$]{$|d| = 4$}
\label{Table4}
\begin{center}
{\renewcommand{\arraystretch}{1.2}
\begin{tabular}{|c|c|c|c|c|c|}
         \hline
         $n$
         & $\psi_1^4$
         & $\psi_1^2\psi_2\psi_3$
         & $\psi_1^2\psi_2^2$
         & $\psi_1\psi_2^3$
         & $\psi_1\psi_2\psi_3\psi_4$
         \\ \hline
         7 
         &
         1
         &
         12
         &
         6
         &
         4
         &
         24
         \\ 
         8 
         &
         $u_1$
         &
         $-40u_1$
         &
         $-10u_1$
         &
         $-5u_1$
         &
         $-120u_1$
         \\ \hline
\end{tabular} }
\end{center}
\end{table}

\renewcommand{\thetable}{A.5}
\begin{table}[H] 
\caption[$|d| = 5$]{$|d| = 5$}
\label{Table5}
\begin{center}
{\renewcommand{\arraystretch}{1.2}
\begin{tabular}{|c|c|c|c|c|c|c|c|}
         \hline
         $n$
         & $\psi_1^5$
         & $\psi_1^3\psi_2\psi_3$
         & $\psi_1\psi_2^2\psi_3^2$
         & $\psi_1^2\psi_2\psi_3\psi_4$
         & $\psi_1^2\psi_2^3$
         & $\psi_1\psi_2^4$
         & $\psi_1\psi_2\psi_3\psi_4\psi_5$
         \\ \hline
         8 
         &
         1
         &
         20
         &
         30
         &
         60
         &
         10
         &
         5
         &
         120
         \\ \hline
\end{tabular} }
\end{center}
\end{table}

\FloatBarrier

\end{document}